\documentclass[fontsize=11pt,toc=bib]{scrartcl}

\setkomafont{title}{\sffamily}
\setkomafont{pagehead}{\sffamily}
\setkomafont{pagenumber}{\sffamily}
\setkomafont{author}{\sffamily}
\addtokomafont{author}{\vspace*{2em}\large}

\setkomafont{dedication}{\footnotesize\raggedleft\itshape}
\usepackage[automark]{scrlayer-scrpage}
\ihead{\headmark}
\ohead{\pagemark}
\usepackage[
  paper=a4paper,
  left=80pt,
  right=80pt,
  top=86pt,
  bottom=56pt,
  bindingoffset=0mm,
  includefoot,
  includehead,
  foot=16.5pt,
  head=16.5pt
]{geometry}

\DeclareMathAlphabet{\pazocal}{OMS}{zplm}{m}{n}
\DeclareMathAlphabet{\mymathbb}{U}{bbold}{m}{n}
\DeclareMathAlphabet{\dutchcal}{U}{dutchcal}{b}{n}

\usepackage{bm}
\usepackage{floatrow}
\usepackage{longtable}
\usepackage{pdflscape}
\usepackage{tcolorbox}
\usepackage{comment}
\usepackage{enumitem}
\usepackage{mathtools}
\usepackage{fontawesome}
\usepackage{booktabs}
\usepackage{indentfirst}
\usepackage{lipsum}
\usepackage[english]{babel}
\usepackage{amsmath} 
\usepackage{amsthm}
\usepackage{amstext}
\usepackage[sfdefault]{FiraSans}
\usepackage[charter,expert]{mathdesign}
\usepackage[scaled=.96]{XCharter}
\usepackage{setspace}
\usepackage{nicefrac}
\usepackage[dvipsnames]{xcolor}
    \definecolor{linkcolor}{HTML}{0D6A9E}
    \definecolor{3blue}{HTML}{0072B2}
    \definecolor{3green}{HTML}{009E73}
    \definecolor{3ochre}{HTML}{E69F00}
    \definecolor{3yellow}{HTML}{F0E442}
    \definecolor{3cyan}{HTML}{56B4E9}
    \definecolor{3red}{HTML}{D55E00}
    \definecolor{3pink}{HTML}{CC79A7}
    \definecolor{2blue}{HTML}{1A85FF}
    \definecolor{2red}{HTML}{D41159}
\usepackage{graphicx}
    \graphicspath{ {./gfx/} }
\usepackage{caption}
\usepackage{tensor}
\usepackage{tikz-cd}
\tikzcdset{arrow style=tikz}
\tikzset{vertl/.style={anchor=south, rotate=90, inner sep=1mm}}
\tikzset{vertr/.style={anchor=south, rotate=-90, inner sep=1mm}}
\usepackage[hang,flushmargin]{footmisc} 
\usepackage{hyperref}
    \hypersetup{
    colorlinks=true,
    linktoc=all,
    linkcolor=linkcolor,
    citecolor=linkcolor,
    filecolor=magenta,      
    urlcolor=linkcolor,,
    pdfauthor=Davide Dal Martello,
    pdftitle=Cluster topograph
    }
\usepackage{cleveref}

\newtheoremstyle{komait}
      {\topsep}   
      {\topsep}   
      {\itshape}  
      {}       
      {\bfseries\sffamily} 
      {}         
      {5pt plus 1pt minus 1pt} 
      {\thmname{#1}\thmnumber{ #2}\thmnote{{\normalfont\,(#3)}}} 
    \newtheoremstyle{komanormal}
      {\topsep}   
      {\topsep}   
      {\rmfamily}  
      {0pt}       
      {\bfseries\sffamily} 
      {}         
      {5pt plus 1pt minus 1pt} 
      {\thmname{#1}\thmnumber{ #2}\thmnote{{\normalfont\,(#3)}}}          
\theoremstyle{komait}
    \newtheorem{theorem}{Theorem}[section]
    \newtheorem{corollary}{Corollary}[section]
    
    \newtheorem{definition}{Definition}[section]

\theoremstyle{komanormal}
    \NewCommandCopy{\proofqedsymbol}{\qedsymbol}
    \newcommand{\openboxthickness}{0.4pt} 
        \renewcommand{\openbox}{\leavevmode
          \hbox to.77778em{%
            \hfil\vrule width \openboxthickness
            \vbox to.675em{%
              \hrule width \dimexpr.675em-2\dimexpr\openboxthickness height \openboxthickness
              \vfil
              \hrule height\openboxthickness
            }%
            \vrule width \openboxthickness
            \hfil
          }%
        }
    \renewcommand{\openboxthickness}{.675pt}
        \AtBeginEnvironment{proof}{\renewcommand{\qedsymbol}{\proofqedsymbol}}
    \newtheorem{example}{Example}[section]
        \AtBeginEnvironment{example}{\pushQED{\qed}\renewcommand{\qedsymbol}{$\triangle$}}
        \AtEndEnvironment{example}{\popQED\endexample}
    
        \AtBeginEnvironment{exercise}{\pushQED{\qed}\renewcommand{\qedsymbol}{$\triangle$}}
        \AtEndEnvironment{exercise}{\popQED\endexample}
    \newtheorem{remark}{Remark}[section]
        \AtBeginEnvironment{remark}{\pushQED{\qed}\renewcommand{\qedsymbol}{$\triangle$}}
        \AtEndEnvironment{remark}{\popQED\endexample}

\newcommand{\G}[1]{\pazocal{G}(\!#1\!)}
\newcommand{\e}[1]{e(\!#1\!)}

\newcommand{\naturali}{\mathbb{N}}
\newcommand{\integer}{\mathbb{Z}}
\newcommand{\rational}{\mathbb{Q}}
\newcommand{\real}{\mathbb{R}}
\newcommand{\complex}{\mathbb{C}}

\begin{document}

\title{\usekomafont{subtitle}\LARGE\vspace{-5em}Cluster topography}
\author{\raisebox{-.5ex}{\href{https://orcid.org/0000-0002-4975-8774}{\includegraphics[height=15pt]{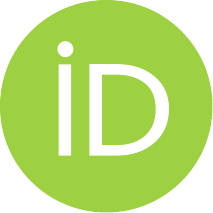}}}\hspace{.5em}Davide DAL MARTELLO\footnote{\hspace{.4em}Department of Mathematics, University of Padua, Via Trieste 63, 35121 Padova, IT\\\faEnvelopeO\hspace*{.5em}davide.dalmartello@unipd.it, \href{mailto:contact@davidedalmartello.com}{contact@davidedalmartello.com}}}
\date{\vspace{-1em}}

\maketitle

\begin{abstract}\noindent
Using the LP algebraic toolkit, Conway's original topograph is rethought of as a cluster construction, paving the way for a wider topography based on mutation-type local rules. As a remarkable application of such cluster-driven upgrade, both the process of analytic continuation for Painlevé VI and the reduction algorithm for quadratic forms are endowed with the Laurent phenomenon. En passant, the rattlesnake is defined so to complete the bijection between snake graphs and rationals to the whole of $\mathbb{Q}$.
\end{abstract}

\noindent
\begin{center}{\vspace{-.5em}\small{\textbf{\textsf{Keywords}}\hspace{.5em}topograph, cluster algebra, Painlevé VI, quadratic form.}}
\end{center}

\renewcommand{\contentsname}{\textcolor{linkcolor}{Contents}}

\tableofcontents

\section{Introduction}

The \textbf{Markov cluster algebra} is a foundational example in the ever more ubiquitous combinatorial framework of mutations. It is driven by the quiver of order 3 having singleton as mutation class (the \textbf{Markov quiver} in \Cref{fig:Markov}), namely the one quiver with three vertices that is invariant, up to isomorphism, under any mutation.
\begin{figure}[!ht]
    \centering
    \includegraphics[width=0.175\linewidth]{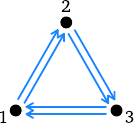}
    \caption{The archetype of mutation-invariant quivers: the Markov quiver. Any mutation behave as the identity up to orientation reversal, e.g. the quiver isomorphism $(1,3)\in\mathfrak{S}_3$.}
    \label{fig:Markov}
\end{figure}
Indeed, mutation is \emph{uniformly} defined at any cluster $(x_1,x_2,x_3)$ by
\begin{equation}\label{Markov-mut}
    \mu_i(x_i)=x_i'=\frac{x_j^2+x_k^2}{x_i}, \qquad i \neq j \neq k.
\end{equation}
Therefore, the algebra is manifestly finite-mutation type---but far from finite type, being not even finitely generated. In particular, its exchange graph is the three-regular tree $\mathbb{T}_3$.
In topological terms, this algebra is surface-type in that the Markov quiver corresponds to the once-punctured torus. Turning the hyperbolic structure on, lengths of closed geodesics give a geometric realization of Markov numbers at unital specialization of Teichm\"uller coordinates, mirroring the well-known algebraic realization given by selecting $(x^0_1,x^0_2,x^0_3)=(1,1,1)$ as initial cluster. Indeed, mutation \eqref{Markov-mut} preserves the \textbf{Markov equation}
\begin{equation}\label{Markov-eq}
    x_1^2+x_2^2+x_3^2=3x_1x_2x_3,
\end{equation}
a rare phenomenon we argue to be best framed in topographic terms.

Originally introduced by Conway \cite{Conway1997}, a \textbf{topograph} $\pazocal{T}$ is a planar embedding of $\mathbb{T}_3$ endowed with integral face labels in arithmetic progression:
displaying $\pazocal{T}$ locally as a compass (\Cref{fig:compass}),
\begin{equation}
    n-(w+e)=(w+e)-s
\end{equation}
must hold in $\integer$ anywhere.
\begin{figure}[!t]
    \centering
    \includegraphics[width=3.25em]{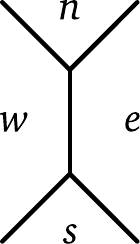}
    \caption{The local shape of a topograph, manifesting the full symmetry group $PGL_2(\integer)$: with respect to presentation \eqref{EUV}, $U$ acts as the vertical edge reflection while $E$ and $V$ as the rotation pivoted in, respectively, the edge center (order $2$) and a vertex (order $3$).}
    \label{fig:compass}
\end{figure}
This local rule can be rewritten as a recursive recipe that allows to construct the whole topograph starting from any vertex $\{n,e,w\}$:
\begin{equation}\label{Conway-rule}
    n+s=2(w+e) \ \iff \ s=2(w+e)-n.
\end{equation}
In particular, being preserved by the latter equality meant as the substitution $n\mapsto s$, the quantity
\begin{equation}
    I_{\mathrm{Con}}(w,n,e):=w^2+n^2+e^2-2(wn+ne+ew)
\end{equation}
is a topographic invariant: its integral value is the same at each and every vertex of $\pazocal{T}$.
As it happens, Markov numbers can be thought of as face labels of a topograph (\Cref{fig:Markovtop}) whose generating local rule is none other than Markov's original Vieta involution:
\begin{equation}\label{Markov-rule}
    n+s=3ew \ \iff \ s=3ew-n.
\end{equation}
In particular, the invariant reads now as
\begin{equation*}
    I_{\mathrm{Mar}}(w,n,e):=w^2+n^2+e^2-3wne
\end{equation*}
which, at zero value, recovers the Markov equation.
As anticipated, the cluster genesis of Markov numbers is, in fact, one of topography: for $x_2$ and $s$ acting as north label and mutated variable respectively without loss of generality,    
\begin{equation}\label{Vieta-to-cluster}
    s=3ew-n\overset{\eqref{Markov-eq}}{=}\frac{w^2+n^2+e^2}{n}-n=\frac{w^2+e^2}{n}.
\end{equation}
In other words, mutation and Vieta involution are equivalent once restricted to the Markov affine cubic \eqref{Markov-eq}. 
In this paper, we show that the converse is true: provided a combinatorial upgrade, topography is the stuff of mutation. In particular, Conway's original construction rooted in the theory of binary quadratic forms gets rebooted in cluster terms.   

More generally understood as a visualization of $PGL_2(\integer)$-dynamics, a topograph can be defined for any affine cubic surface playing the role of its vertex-invariant, provided a suitable constructing local rule is defined.
\begin{figure}[!t]
    \centering
    \includegraphics[width=0.525\linewidth]{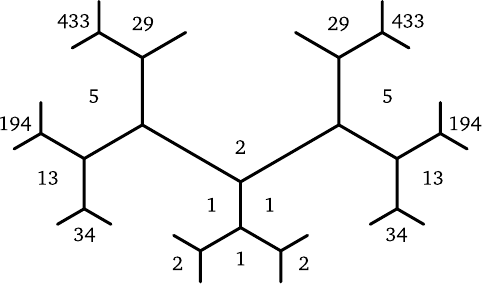}
    \caption{Beginning of the \textbf{Markov topograph}, whose vertices correspond to (positive) solutions of \eqref{Markov-eq}. Being the initial vertex $\{1,1,1\}$ invariant for the whole of $\mathfrak{S}_3$, the labeling is highly symmetric.}
    \label{fig:Markovtop}
\end{figure}
\begin{figure}[!t]
    \centering
    \includegraphics[width=0.875\linewidth]{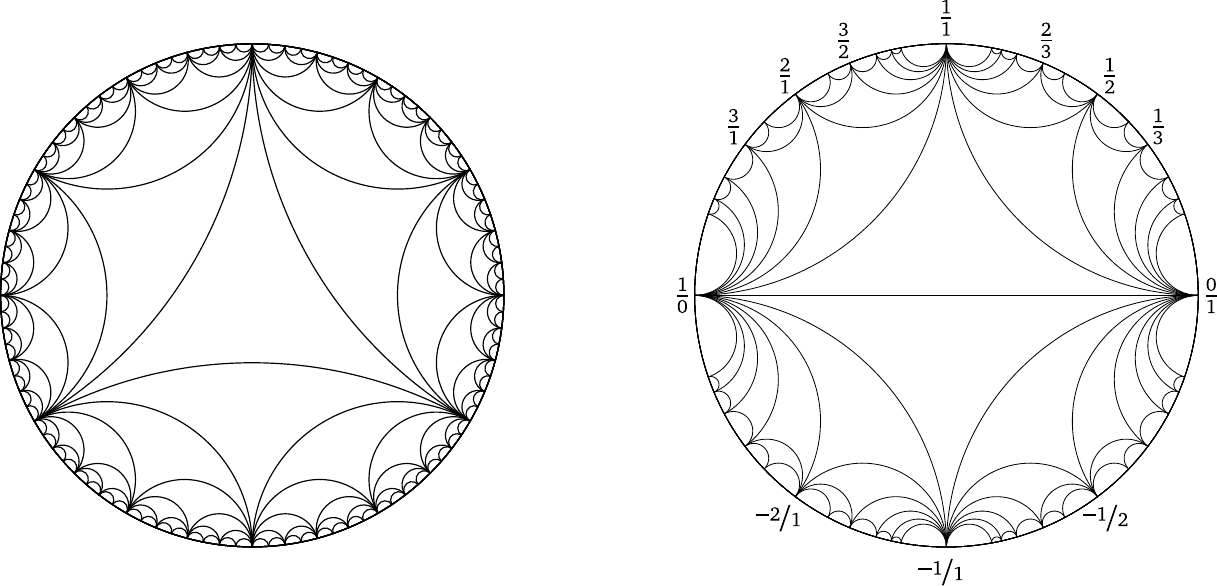}
    \caption{The circular Farey tessellation---topological (triangular) on the left, hyperbolic (quadrangular) on the right with Farey fractions on the boundary $\real$---gives a standard embedding of $\mathbb{T}_3^*$, the dual graph of the $3$-regular tree, on the disk.}
    \label{fig:Farey}
\end{figure}
We push such generalization to the limit by (parametric) saturation of the lower degree terms in the Markov equation:
\begin{equation}\label{master-eq}
    x_1^2+x_2^2+x_3^2+\delta_1x_2x_3+\delta_2x_3x_1+\delta_3x_1x_2+\sigma_1x_1+\sigma_2x_2+\sigma_3x_3+\zeta=\tau x_1x_2x_3.
\end{equation}
This cubic, referred to as the \textbf{master equation}, induces the invariant
\begin{equation}\label{master-inv}
    I_{\mathrm{Mas}}(w,n,e):=w^2+n^2+e^2+\delta_1ne+\delta_2ew+\delta_3wn+\sigma_1w+\sigma_2n+\sigma_3e-\tau wne
\end{equation}
which, in turn, is readily checked to be preserved by the local (parametric) rule
\begin{equation}\label{master-rule}
    n+s=\tau we-\delta_1e-\delta_3w-\sigma_2 \ \iff \ s=\tau ew-\delta_1e-\delta_3w-\sigma_2-n
\end{equation}
that admits a cluster genesis in the wider sense of \textbf{LP algebra}.
In a nutshell, this generalization of classical cluster algebra shifts the focus from exchange matrices to exchange polynomials, and quivers cease to control the LP mutation. Nevertheless, exchange polynomials are a signature conserved datum of our cluster dynamics (§\ref{sec:master}), which can thus be thought of as the ``maximally Laurent'' extension of the classical dynamics \eqref{Markov-mut} controlled by the Markov quiver.
As a crucial consequence, the topograph's recursive construction is endowed with the characteristic \textbf{Laurent phenomenon}: face labels are expressible as Laurent polynomials with respect to the three ones $\{n,e,w\}$ at any vertex.

Two remarkable applications showcase the impact of this new \textbf{cluster topography}: the nonlinear monodromy of the sixth Painlevé equation as an example of a complex topograph, the reduction of quadratic forms as a Diophantine example.
In the former, clusters $(w,n,e)\in\complex^3$ correspond to branches of a given solution and suitable pairs of mutations realize the structure of analytic continuation.
In the latter, clusters $(w,n,e)\in\integer^3$ correspond to binary quadratic forms of a given equivalence class and suitable sequences of mutations realize the reduction algorithm.
\begin{theorem}\label{thm:intro}
    The following processes exhibit the Laurent phenomenon:
    \begin{itemize}
        \item[$\mathrm{(A)}$] the analytic continuation of solutions to Painlevé VI,
        \item[$\mathrm{(B)}$] the reduction of binary integral quadratic forms.
    \end{itemize}
\end{theorem}

In particular, part B requires an extension of the standard connection between so-called snake graphs and continued fraction expansions that is of independent interest. More precisely \cite{CS2018}, a planar graph $\pazocal{G}[a_1,a_2,\ldots,a_n]$ can be attached to any \emph{positive} continued fraction ($a_i\geq1$) in such a way that the latter results as the quotient of the number of perfect matchings of $\pazocal{G}[a_1,a_2,\ldots,a_n]$ and $\pazocal{G}[a_2,\ldots,a_n]$. For the cluster reduction algorithm to succeed in all possible cases of a fraction $\pm\frac{r}{s}=[a_1,a_2,\ldots,a_n]\in\rational$, positiveness must be lifted by adopting the \textbf{rattlesnake} $(\pazocal{G},e)$ instead.
Crucially, this marked one-tile extension of $\pazocal{G}[a_1,\ldots,a_n]$ can tell the fraction's sign and magnitude all the while counting its perfect matchings recovers $r+s$, offering a generalization of the original bijection \cite[\S4]{CS2018} to the \emph{whole} of $\rational$.

The present paper is organized as follows:

\S\ref{sec:master} develops the general theory of cluster topographs underpinned by the master equation and the associated constructive LP local rule;

\S\ref{sec:specializations} delves into the pair of applications, starting with Painlevé VI (\S\ref{sec:PVI}) and later turning Conway's original topography cluster (\S\ref{sec:quadratic}), culminating at last in a mutation-type reduction algorithm based on the rattlesnake (\S\ref{sec:clustred}). 

\paragraph{Acknowledgments}
\!\!\!The author is truly grateful to Daniel Labardini Fragoso for both suggesting this line of research and the many insightful discussions, which were enriched by the contributions of Javier De Loera Chávez, Azzurra Ciliberti, and Viola Conte.
In particular, specialization \eqref{q-mutation} of the ``master setting'' answers in the positive a question raised by Labardini Fragoso.
This project was funded by Fondazione Cariparo [C93C22008360007] via University of Padua.

\section{Master is cluster}\label{sec:master}

This section lays the foundation of a cluster topography driven by the master equation \eqref{master-eq}, starting with the definition of a rational ``master'' local rule.
Mimicking computations \eqref{Vieta-to-cluster} from our inspiring example, at once we merge and generalize Conway and Markov rules with a striking cluster flavor: on the master cubic with $x_2$ and $s$ acting again as north label and mutated variable respectively,
\begin{equation}\label{master-rule2}
\begin{aligned}
    n=\tau we-\delta_1e-\delta_3w-\sigma_2-s&\overset{\eqref{master-eq}}{=}\frac{w^2+n^2+s^2+\delta_2ew+\sigma_1w+\sigma_3e+\zeta}{n}-n=\\&\,=\frac{w^2+s^2+\delta_2ew+\sigma_1w+\sigma_3e+\zeta}{n}.
\end{aligned}
\end{equation}
Indeed, specialization
\begin{equation}\label{Mar-spec}
    \sigma_i=\delta_i=0,\ \tau=3
\end{equation}
recovers the Vieta involution while specialization
\begin{equation}\label{Con-spec}
    \sigma_i=0,\ \delta_i=-2,\ \tau=0
\end{equation}
recovers the arithmetic progression.
Notice that rule \eqref{master-rule}, which can be thought of as Vieta-type for a cubic polynomial \emph{at most quadratic} in every variable, is obtained by collecting a (negative) factor of $n$ in the master invariant:
\begin{equation}
    w^2+n^2+\delta_2ew+\sigma_1w+\sigma_3e-n\underbrace{(\tau ew-\delta_1e-\delta_3w-\sigma_2-n)}_s.
\end{equation}

Here goes the core question: 
\begin{quote}
Is the last rational expression in \eqref{master-rule2} a mutation formula?
\end{quote}
Validation comes from it both being involutive, as the archetype of an elementary combinatorial operation, and exhibiting the signature Laurent phenomenon.

\paragraph{Involutiveness} Checking the formula undoes itself is a straightforward computation: applying the rule twice,
\begin{equation}
    n \ \mapsto \ \frac{w^2+s^2+\delta_2ew+\sigma_1w+\sigma_3e+\zeta}{n} \ \mapsto \ \frac{w^2+s^2+\delta_2ew+\sigma_1w+\sigma_3e+\zeta}{\frac{w^2+s^2+\delta_2ew+\sigma_1w+\sigma_3e+\zeta}{n}}=n.
\end{equation}

\paragraph{Laurentness} As the would-be exchange polynomial $w^2+s^2+\delta_2ew+\sigma_1w+\sigma_3e+\zeta$ is far from the binomial fitting classical cluster algebra, we must rely on the theory of Laurent Phenomenon (LP) algebras, whose namesake property is stated in \cite[Theorem 5.1]{LP2016}. In particular, proving our candidate exchange polynomial is really an invariant of the LP mutation suffices to establish the Laurent phenomenon.

In LP jargon, our seed reads from the master equation as
\begin{equation}
    t=(\mathbf{x},\mathbf{F})
\end{equation}
for $\mathbf{x}=\{x_1,x_2,x_3\}$ and $\mathbf{F}=\{F_1,F_2,F_3\}$ the triple of irreducible exchange polynomials
\begin{equation}
    F_i=x_j^2+x_k^2+\delta_ix_jx_k+\sigma_jx_j+\sigma_kx_k+\zeta, \qquad i\neq j \neq k \in \integer_3.
\end{equation}
Given the cyclic symmetry, the triplet of mutations can be still described uniformly on $t$ by $\mu_i$, $i=1,2,3$. Since both $F_j$ and $F_k$ depend on $x_i$, $\hat{F}_i=F_i$ for all $i=1,2,3$ and
\begin{equation}
    \mu_i(t)=t'=(\mathbf{x}',\mathbf{F}') 
\end{equation}
with
\begin{equation}
    \begin{cases}
        x_i'=\dfrac{x_j^2+x_k^2+\delta_ix_jx_k+\sigma_jx_j+\sigma_kx_k+\zeta}{x_i},\\
        x_j'=x_j,\\[.4em]
        x_k'=x_k,\\[.7em]
    \end{cases}
\end{equation}
and, a priori, only $F_i'=F_i$ preserved.
For the master equation to be preserved by any mutation, we must further have $F_y'=F_y$ and $F_z'=F_z$.
Namely, mutation must be uniformly defined across \emph{all} clusters much as it happened with the Markov cluster algebra.
Following LP rules, first define
\begin{align*}
    G_j=&\ F_j\raisebox{-.275em}{\big|}_{x_i\leftarrow\frac{F_i|_{x_j\leftarrow0}}{x_i'}}=\\
    =&x_k^2+\left(\frac{x_k^2+\sigma_kx_k+\zeta}{x_i'}\right)^2+\delta_jx_k\left(\frac{x_k^2+\sigma_kx_k+\zeta}{x_i'}\right)+\sigma_kx_k+\sigma_i\left(\frac{x_k^2+\sigma_kx_k+\zeta}{x_i'}\right)+\zeta=\\=&(x_k^2+\sigma_kx_k+\zeta)\underbrace{\left(1+\frac{\sigma_i}{x_i'}+\frac{\delta_jx_k}{x_i'}+\frac{x_k^2+\sigma_kx_k+\zeta}{x_i'^2}\right)}_{H_j}
\end{align*}
and
\begin{align*}
    G_k=&\ F_k\raisebox{-.275em}{\big|}_{x_i\leftarrow\frac{F_i|_{x_k\leftarrow0}}{x_i'}}=\\\allowdisplaybreaks
    =&\left(\frac{x_j^2+\sigma_jx_j+\zeta}{x'_i}\right)^2+x_j^2+\delta_k\left(\frac{x_j^2+\sigma_jx_j+\zeta}{x_i'}\right)x_j+\sigma_i\left(\frac{x_j^2+\sigma_jx_j+\zeta}{x_i'}\right)+\sigma_jx_j+\zeta=\\\allowdisplaybreaks=&(x_j^2+\sigma_jx_j+\zeta)\underbrace{\left(1+\frac{\sigma_i}{x_i'}+\frac{\delta_kx_j}{x_i'}+\frac{x_j^2+\sigma_jx_j+\zeta}{x_i'^2}\right)}_{H_k}.
\end{align*}
Given that
\begin{equation}
\begin{aligned}
    H_j=\frac{x_k^2+x_i'^2+\delta_jx_kx_i'+\sigma_kx_k+\sigma_ix_i'+\zeta}{x_i'^2}=\frac{F_j|_{x_i\leftarrow x_i'}}{x_i'^2},\\
    H_k=\frac{x_i'^2+x_j^2+\delta_kx_i'x_j+\sigma_ix_i'+\sigma_jx_j+\zeta}{x_i'^2}=\frac{F_k|_{x_i\leftarrow x_i'}}{x_i'^2},
\end{aligned}    
\end{equation}
one naively chooses $M=x_i'^2$ as absorbing monomial to confirm that both $F'_j=M\cdot H_j=F_j(\mathbf{x}')$ and $F'_k=M\cdot H_k=F_k(\mathbf{x}')$. In particular, notice this constructive proof works for any specialization of the parameters $\zeta,\sigma_i,\delta_j,\tau$.

\begin{figure}[!t]
    \centering
    \includegraphics[width=0.5\linewidth]{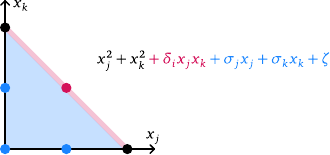}
    \caption{Interpretation à la Newton for the relative generalization of the Markov exchange polynomials: Checkov-Shapiro in {\color{2red}red}, Laurent Phenomenon in {\color{2blue}blue}. CS generalization saturates the Newton polygon of the standard binomial $x_j^2+x_k^2$ with the remaining ($1$-parametric) quadratic monomial {\color{2red}$\delta_ix_jx_k$}, while the LP framework further allows all ($1$-parametric) monomials of lower degree. The resulting monomials exhibit a flag-like containment: LP monomials span the whole face of the triangle, and contain CS monomials on the hypotenuse line segment---itself containing the classical cluster binomial as endpoints.}
    \label{fig:Newton}
\end{figure}

\begin{remark}
In surface cluster algebraic terms, the topology can only control the squared monomials of the exchange polynomial. Generalizing classical cluster algebra, Chekhov and Shapiro can give the mixed quadratic term a geometric encoding via $\lambda$-lengths by allowing orbifold points on the once-punctured torus \cite{CS2013}. We postpone to future work the investigation of a finer geometric setting able to further control the lower degree terms handled by the LP machinery. As observed in \cite[§7.3]{LP2016}, any
exchange polynomial of a cluster algebra (being just a binomial) has a line segment as Newton polytope. Essentially, Chekhov-Shapiro generalized cluster algebras $\pazocal{A}_{CS}$ are those Laurent Phenomenon algebras $\pazocal{A}_{LP}$ for which
there exists a cluster algebra $\pazocal{A}$ and a bijection between the seeds of $\pazocal{A}$ and $\pazocal{A}_{LP}$ that preserves the Newton polytope. In light of this interpretation, \Cref{fig:Newton} visualizes the containment relation that arises for the three types of algebras when the graph-theoretic information is given by very same Markov quiver.
\end{remark}

\paragraph{Topography}
Conway's beautiful concept of a topograph for quadratic forms, whose anticipated cluster nature is explicitly formalized in §\ref{sec:quadratic}, can be given a ``master'' manifestation. Each cluster $(x_1,x_2,x_3)$ solving \eqref{master-eq} is visualized by a trivalent vertex of the $3$-regular tree $\mathbb{T}_3$, with the three connected solutions being precisely the clusters obtained via mutations $\mu_1$, $\mu_2$, and $\mu_3$. Indeed, assuming $\mathbf{x}$ solves the master equation, $\mathbf{x}'=\mu_i(\mathbf{x})$ remains a solution: for $i\neq j \neq k \in \integer_3$, 
\begin{multline}\label{sol-preserved}
    x_i^{\prime2}+x_j^2+x_k^2+\delta_ix_jx_k+\delta_jx_kx'_i+\delta_kx'_ix_j+\sigma_ix'_i+\sigma_jx_j+\sigma_kx_k+\zeta-\tau x'_ix_jx_k=\\=\dfrac{F_i}{x_i^2}(x_i^2+x_j^2+x_k^2+\delta_ix_jx_k+\delta_jx_kx_i+\delta_kx_ix_j+\sigma_ix_i+\sigma_jx_j+\sigma_kx_k+\zeta-\tau x_ix_jx_k)
\end{multline}
and last factor vanishes.
In other words, the topograph's skeleton is none other than the \textbf{exchange graph}, whose dual object is the \textbf{cluster complex} \cite[\S3.6]{LP2016}. In the latter, clusters are visualized by maximal simplices that are adjacent whenever a mutation maps between the corresponding solutions. In our cubic setting, a maximal simplex is a triangle and the cluster complex can be compactly visualized by the \emph{topological} \textbf{Farey tessellation} of the disk (\Cref{fig:Farey}) whose central reference triangle plays the role of initial cluster (\Cref{fig:clustop}).
\begin{figure}[!t]
    \centering
    \includegraphics[width=0.3\linewidth]{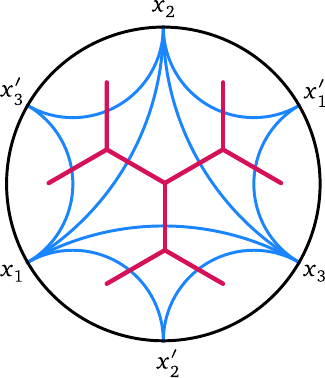}
    \caption{{\color{2red}Topograph} as exchange graph, (topological) {\color{2blue}Farey tessellation} as cluster complex. Labels of the $\infty$-gons in the former are vertices in the latter 
    and belong to the same algebraic structure the cubic is defined within, e.g. the field $\complex$ or the ring $\integer$.}
    \label{fig:clustop}
\end{figure}
In particular, the triangle's full symmetry group is the dihedral $D_3\simeq\mathfrak{S}_3$. Whether any of these symmetries lifts to the solving triplet---and is thus compatible with Laurentness---depends on the genericity of the cubic's parameters as showcased in \Cref{ex:genericity}.
Nevertheless, mutations are always allowed in the form of motion to adjacent triangles by edge-reflection, giving a copy of $\integer_2*\integer_2*\integer_2$ meant as a subgroup of the topograph's full symmetry group $PGL_2(\integer)$. 

Assuming the master equation is integer-valued, namely has $\integer^8$ as parameter space and integral solutions, the analogy with Conway's original topography is even stronger.
Indeed, the LP mutation preserves integrality when restricted to solutions, being equivalent to rule \eqref{master-rule} on the master equation: assuming $\mathbf{x}\in\integer^3$ solves \eqref{master-eq} and $i\neq j \neq k \in \integer_3$, 
\begin{equation*}
\begin{aligned}
    x_i'=\mu_i(x_i)&=\dfrac{x_j^2+x_k^2+\delta_ix_jx_k+\sigma_jx_j+\sigma_kx_k+\zeta}{x_i}\overset{\eqref{master-eq}}{=}\frac{x_i(\tau x_jx_k-x_i-\delta_jx_k-\delta_kx_j-\sigma_i)}{x_i}=\\&=\tau x_jx_k-x_i-\delta_jx_k-\delta_kx_j-\sigma_i\in\integer.
\end{aligned}
\end{equation*}
Such integral-compatibility gives an entry point for a Diophantine take on cluster topography: find the subset of the parameter space $\integer^8$ for which the master equation admits integral solutions, verify the LP mutation produces them all, study if the resulting integral topography is ``hydrographic'' as found by Conway, and so on. We plan to explore this specialization in future work.

\begin{example}\label{ex:genericity}
    Inspecting the respective invariants, both Markov and Conway specializations admit a lift of the full $\mathfrak{S}_3$. Indeed, triplet of Markov numbers are customarily considered as unordered while $\mathfrak{S}_3<PGL_2(\integer)$, the latter being the natural equivalence group for binary integral quadratic forms (cf. \S\ref{sec:equiv}). Consider now the following specialization of the master equation:
    \begin{equation*}
        x_1^2+x_2^2+x_3^2+\sigma(x_1+x_2)+x_3=0.
    \end{equation*}
    By visual inspection, one expects only $(12)\in\mathfrak{S}_3$ to fit the resulting cluster topography.
    Indeed, $\mu_1(12)=(12)\mu_2$ while $\mu_3$ and $(12)$ simply commute. E.g.,
    \begin{align*}
        (x_1,x_2,x_3) \ &\overset{\mu_2}{\longmapsto} \ \left(x_1,\frac{\sigma x_1+x_1^2+x_3^2+x_3}{x_2},x_3\right) \ \overset{(12)}{\longmapsto} \ \left(\frac{\sigma x_1+x_1^2+x_3^2+x_3}{x_2},x_1,x_3\right) \ \longmapsto \\ &\overset{\mu_3}{\longmapsto} \ \left(\frac{\sigma x_1+x_1^2+x_3^2+x_3}{x_2},x_1,\frac{x_3^4+\ldots}{x_2^2x_3}\right)=(12)\mu_3\mu_2(x_1,x_2,x_3)=\mu_3\mu_1(x_2,x_1,x_3).
    \end{align*}
    Moreover,
    \begin{equation*}
        \mu_3(23)\mu_2(x_1,x_2,x_3)=\left(x_1,x_3,x_2\frac{ x_1^2+x_3^2+\sigma x_1+\sigma x_3}{x_1^2+x_3^2+\sigma x_1+x_3}\right)
    \end{equation*}
    generates a truly rational expression, confirming our expectations in full. 
\end{example}

\begin{remark}
    Local rules, being just algebraic, directly preserve the vertex-invariant so that its value does not appear explicitly in their polynomial expression. On the contrary, were $\zeta$ missing from our LP mutation rule, only
    \begin{equation*}
        I_{\mathrm{Mas}}(\mu_i(\mathbf{x}))=\frac{F_i(\mathbf{x})}{x_i^2}I_{\mathrm{Mas}}(\mathbf{x})
    \end{equation*}
    would hold, failing to map solutions to solutions. Cf. \eqref{sol-preserved}, where the extra factor $F_i/x_i^2$ does not prevent the mutated triple from solving the equation. This phenomenon didn't manifest in the Markov special case \eqref{Vieta-to-cluster}, whose invariant's value is needed to be \emph{precisely} zero. On the flip side, the cubic term (parametric in $\tau$) only contributes to the local rule and disappears entirely from the cluster language. Like a wave-particle duality, the polynomial local rule and the rational mutation formula, respectively first and last in \eqref{master-rule2}, coexist to yield complementary understandings of their common discrete topographic dynamics.    
\end{remark}

\begin{remark}
When $x_i\rightarrow1$, the Markov equation reduces to the ``Kronecker equation'': 
\begin{equation}\label{Kronecker-eq}
    x_j^2+x_k^2+1=3x_jx_k.
\end{equation}
Indeed, the Vieta involution degenerates to
\begin{equation*}
    x_j \mapsto 3x_k-x_j=\frac{x_k^2-1}{x_j},
\end{equation*}
matching the exchange binomial driven by the \emph{mutation-invariant} Kronecker quiver $K_2$. As expected, positive integral solutions are those neighboring any of the three $\infty$-gons with label 1 in the Markov topograph.
More generally, the limit $x_i\rightarrow1$ of the master equation reads, after incorporating parameters, as 
\begin{equation}
    x_j^2+x_k^2+\sigma_jx_j+\sigma_kx_k+\zeta=\delta x_jx_k,
\end{equation}
namely a master (Kronecker) equation with exchange polynomials
\begin{align*}
    F_j&=x_j^2+\sigma_jx_j+\zeta,\\
    F_k&=x_k^2+\sigma_kx_k+\zeta,
\end{align*}
for which the Chekhov-Shapiro toolkit suffices.
This phenomenon hints at a more general theory for mutation-invariant quivers, but other attempts (e.g. $A_2$ or $K_3$) seem to suggest otherwise. 
\end{remark}

\section{Two specializations}\label{sec:specializations}

By specialization of the master parameters, this section tailors our emergent cluster topography to the highly influential settings of the Painlevé equations and binary quadratic forms, endowing a core process in each theory with the Laurent phenomenon.

\subsection{$\mathrm{PVI}$ monodromy manifold and analytic continuation}\label{sec:PVI}

 Throughout this subsection, the ordered triple of indices $(i,j,k)$ denotes any cyclic permutation of the $3$-tuple $(1,2,3)$.
 We focus first on the \textbf{Painlevé VI equation}, denoted $\mathrm{PVI}=\mathrm{PVI}(\kappa)$ as being dependent on the quadruple of parameters $(\kappa_1,\kappa_2,\kappa_3,\kappa_4)\in\complex^4$:
\begin{multline}\label{PVI}
y_{tt}=\frac{1}{2}\left({1\over y}+{1\over y-1}+{1\over y-t}\right) y_t^2 -
\left({1\over t}+{1\over t-1}+{1\over y-t}\right)y_t+\\+{y(y-1)(y-t)\over 2t^2(t-1)^2}\left[\kappa_{\!4}^2-\kappa_{\!1}^2 {t\over y^2}+
\kappa_{\!2}^2{t-1\over (y-1)^2}+(1-\kappa_{\!3}^2){t(t-1)\over(y-t)^2}\right].
\end{multline}
By the defining \textbf{Painlevé property}, movable singularities of this second order ODE---which acts as the nonlinear analogue of the Gauss hypergeometric equation---can only be simple poles. In particular, any solution $y(t)$ is precisely branched just at the equation's own critical points $0,1,\infty$, and can be thus analytically continued to the whole universal cover of $\overline{\complex}\setminus\{0,1,\infty\}$.

As anticipated, we tailor the theory established in \Cref{sec:master} so to rethink the process of analytic continuation for PVI as one of cluster combinatorics, beautifully bridging complex analysis with combinatorial algebra.
The key viewpoint is provided by the Riemann-Hilbert correspondence, which turns the (nonlinear) monodromy of PVI into a $\Gamma(2)$-action by polynomial automorphisms on the so-called \textbf{monodromy manifold} 
\begin{equation}\label{mon-mfd}
    x_1x_2x_3+x_1^2+x_2^2+x_3^2-\theta_1x_1-\theta_2x_2-\theta_3x_3+\theta_4=0,
\end{equation}
where $\theta=\theta(a_1,a_2,a_3,a_\infty)$ as
\begin{equation}\label{theta-as-a}
\begin{aligned}
    \theta_i=&a_ia_\infty+a_ja_k,\\
    \theta_4=&a_1a_2a_3a_\infty+a_1^2+a_2^2+a_3^2+a_\infty^2-4,
\end{aligned}
\end{equation}
under correspondence
\begin{equation}
    a_i=2\mathrm{cos}(\pi\kappa_i), \quad a_\infty=-2\mathrm{cos}(\pi\kappa_4).
\end{equation}

The first instance of this ``algebrification'' was given in \cite{DM2000} for the 1-parametric specialization $\mathrm{PVI}(\mu)$ selected by $\kappa_1=\kappa_2=\kappa_3=0$ and $\kappa_4=2\mu-1$. Taking advantage of the equation's isomonodromic genesis, branches of solutions are parametrized by triples $(M_0,M_1,M_t)$ of monodromy matrices---each one measuring the failure of single-valuedness as its corresponding pole is encircled. In turn, this representation space admits a reflection-type coordinatization singled out as the zero set of a special cubic polynomial, further simplifying the parametrization:
\begin{theorem}[{\cite[Theorem 1.4]{DM2000}}]\label{thm:branches}
    The branches of solutions to PVI$(\mu)$ near $p_0\in\overline{\complex}\setminus\{0,1,\infty\}$ are in one-to-one correspondence with equivalence classes of the admissible triples $(x_1,x_2,x_3)$ satisfying
 \begin{equation*}
     x_1^2+x_2^2+x_3^2-x_1x_2x_3-4\mathrm{sin}^2(\pi\mu)=0.
 \end{equation*}
\end{theorem}
In particular, the notion of equivalence is driven by a manifest symmetry of the affine cubic surface (sign-change on pairs of coordinates) while admissibility boils down to having at most one zero coordinate $x_i$. The theorem provides an algebraic sandbox allowing to encapsulate the process of analytic continuation into the following action of $B_3=\langle\beta_1,\beta_2 \ | \ \beta_1\beta_2\beta_1=\beta_2\beta_1\beta_2 \rangle$, the braid group in three strands \cite[Lemma 1.10]{DM2000}:
\begin{equation}\label{braid}
\begin{aligned}
    \beta_i \ : \ (x_i,x_j,x_k) \ &\mapsto \ (-x_i,x_k-x_ix_j,x_j),    
\end{aligned}
\end{equation}
where we defined the additional generator $\beta_3:=\beta_1\beta_2\beta_1^{-1}$ in a cyclic fashion taking advantage of equivalence:
\begin{equation*}
\beta_1\beta_2\beta_1^{-1}(x_1,x_2,x_3)=(-x_2+x_3x_1,-x_1,-x_3)\sim(x_2-x_3x_1,x_1,-x_3).   
\end{equation*}
In fact, $(\beta_1\beta_2)^3=1$ reveals this as an action of the modular group $\Gamma:=\mathrm{PSL}_2(\integer)=B_3/Z(B_3)$ in that $Z(B_3)=\langle(\beta_1\beta_2)^3\rangle$. Moreover, $\beta_i$ actually expresses the \emph{superposition} of analytic continuation and the permutation of finite poles
\begin{equation}
    (u_i,u_j,u_k) \ \mapsto \ (u_j,u_i,u_k), 
\end{equation}
for the conventional choice $(u_1,u_2,u_3)=(0,1,t)$ used in \eqref{PVI}.

For $(x_1,x_2,x_3,a_1,a_2,a_3,a_\infty)=:(\mathbf{x},\mathbf{a})\in\complex^{7}$, formulae \eqref{braid} were extended by Iwasaki \cite{Iwasaki2003} to the general PVI($\kappa$): 
\begin{equation*}
    \beta_k \ : \ (\mathbf{x},\mathbf{a})\ \mapsto \ (\mathbf{x}',\mathbf{a}')
\end{equation*}
now acts on the full monodromy manifold \eqref{mon-mfd} as the identity map except for
\begin{equation}
\begin{cases}
    x'_i=\theta_j(\mathbf{a})-x_j-x_kx_i,\\x'_j=x_i,
\end{cases}\qquad
\begin{cases}
    a'_i=a_j,\\a'_j=a_i.
\end{cases}
\end{equation}
Notice that the action now affects the whole $\complex^7$, namely both sets of coordinates and parameters (cf. \eqref{braid} in which $\mu$ did not appear explicitly), and is singled out by the data it fixes: one monodromy matrix $M_k$ (with its trace $a_k$) together with the same-subscript reflection coordinate $x_k$. This drives our notation for the action's subscripts, which differs from Iwasaki's.

The above extension to the full PVI is still a superimposition; in order to isolate analytic continuation, one should absorb the $\mathfrak{S}_3$-action on the parameter subspace $(a_1,a_2,a_3)\in\complex^3$. Following Iwasaki's topological approach, this can be achieved by restricting to the pure braid (sub)group, i.e., squaring each braid generator to obtain
\begin{equation}
    \beta_k^2 \ : \ (\mathbf{x},\mathbf{a})\ \mapsto \ (\mathbf{x}'',\mathbf{a})
\end{equation}
for
\begin{equation}\label{beta2}
\begin{cases}
    x''_i=\theta_j(\mathbf{a}')-x_i-x_k(\theta_j(\mathbf{a})-x_j-x_kx_i)=\theta_i(\mathbf{a})-(1-x_k^2)x_i-(\theta_j(\mathbf{a})-x_j)x_k,\\x''_j=\theta_j(\mathbf{a})-x_j-x_kx_i,\\x''_k=x_k.
\end{cases}
\end{equation}
Notice that dependence \eqref{theta-as-a} implies $\theta_j(\beta_k(\mathbf{a}))=\theta_i(\mathbf{a})$.
In particular, this is a faithful action of $\Gamma(2)\simeq F_2$ \cite{PR2024}, the level 2 principal congruence subgroup, on a slice of the monodromy manifold selected by \emph{fixing} the local monodromy $\mathbf{a}=(a_1,a_2,a_3,a_\infty)$. Indeed, by direct computation  
\begin{equation*}
    \beta_3^2\beta_1^2\beta_2^2(\mathbf{x})=\mathbf{x}.
\end{equation*}
As any pure braid fixes each strand's endpoints, $\beta_i^2$ fixes all parameters. Moreover, the restriction to $\Gamma(2)$ narrows the focus to the \emph{structure} of analytic continuation which, by the Painlevé property, is indeed described as an action of \emph{loops} in $\pi_1(\overline{\complex}\setminus\{0,1,\infty\})$, the fundamental group of the three-punctured sphere, that measures monodromy.
Summarizing,
\begin{theorem}[{\cite[Theorem 6.5]{Iwasaki2003}}]
    The $\Gamma(2)$-action \eqref{beta2} on the monodromy manifold with fixed local data $a\in\complex^4$ is an explicit representation of the nonlinear monodromy of $\mathrm{PVI(\kappa)}$.
\end{theorem}

As it happens, there is an alternative realization of the very same action in the language of our cluster topography.
Specializing master rule \eqref{master-rule2} to the monodromy manifold \eqref{mon-mfd}, a mutation preserves parameters and all coordinates except the one sharing its subscript, which mutates as
\begin{equation*}
    x_i'=\mu_i(x_i)=\theta_i-x_i-x_jx_k=\frac{x_j^2+x_k^2-\theta_jx_j-\theta_kx_k+\theta_4}{x_i}.
\end{equation*}
It is then easy to realize that
\begin{equation}
    \mu_j = (ij)\beta_k
\end{equation}
for the purely combinatorial permutation
\begin{equation}
    (ij) \ : \ \begin{cases}
    M_i'=M_j,\\M'_j=M_i,\\M'_k=M_k,
\end{cases}
\end{equation}
swapping both local and global monodromy data indexed by $i,j$. This last operation manifests no topological effect and distills precisely the analytic continuation part of $\beta_k$ by compensating its permutation of poles.

Compared against braids, mutations simplify the action in that we can forget entirely about the dependence of the manifold's parameters $\theta$ on local monodromy $a$. On the flip side, squaring the actions now fails to express monodromy, being a mutation involutive by definition. Nevertheless, with little effort one finds inside $\langle\mu_1,\mu_2,\mu_3\ | \ \mu_1^2=\mu_2^2=\mu_3^2=1 \rangle$ a copy of the desired $\Gamma(2)\simeq F_2$ in the subgroup
\begin{equation}
    \langle \mu_{12},\mu_{23},\mu_{31} \ | \ \mu_{12}\mu_{23}\mu_{31}=1 \rangle,
\end{equation}
whose generators act none other than $\beta_k^2$:
\begin{equation}
\mu_{ij}:=\mu_i\mu_j \ : \ 
\begin{cases}
    x_i \ \mapsto \ \theta_i-x_i-\mu_j(x_j)x_k=\theta_i-(1-x_k^2)x_i-(\theta_j-x_j)x_k,\\x_j \ \mapsto \ \mu_j(x_j)=\theta_j-x_j-x_kx_i,\\x_k \ \mapsto \ x_k.
\end{cases}
\end{equation}
In other words, analytic continuation exhibits the (nontrivial) composition of pairs of mutations as elementary building block.
We have thus proven part A of \Cref{thm:intro}: 
\begin{theorem}\label{thm:anconLaurent}
    The procedure of analytic continuation of solutions to the sixth Painlevé equation $\mathrm{PVI}(\kappa)$, for any $\kappa\in\complex^4$, exhibits the Laurent phenomenon: starting from a local solution corresponding to the point $\mathbf{x}^0=(x^0_1,x^0_2,x^0_3)\in(\complex^*)^3$ on the monodromy manifold, any other branch will correspond to a point whose coordinates express as Laurent polynomials of the initial $\mathbf{x}^0$.
\end{theorem}
\begin{figure}[t]
    \centering
    \includegraphics[width=0.325\linewidth]{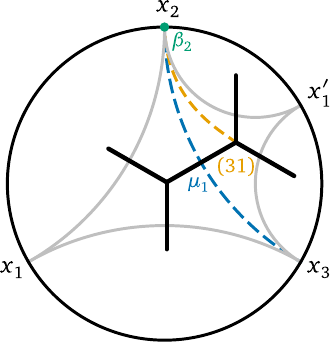}
    \caption{Topographic comparison between a mutation and a braid move. The former is visualized by an edge-reflection (with axis dotted in {\color{3blue}blue}) while the latter manifests as a rotation pivoted in the {\color{3green}green} center. Visually, the two moves are clearly intertwined by the {\color{3ochre}yellow} reflection $(31)$. Such permutation---and thus the rotation---eludes the symmetry group of this ``monodromic'' cluster topography in that all parameters $\theta$ are assumed to be generic.} 
    \label{fig:braidvsmut}
\end{figure}
We can visually interpret the process of analytic continuation in topographic terms: triangles correspond to monodromy triples, flips to analytic continuation and pairs of flips to genuine branch jumps.
Starting from any local solution, all its branches can be reached through sequences of such pairs and, as a whole, they populate half of the topograph's vertices---which can be though of as one part of a bipartite 3-regular tree. In particular, vertices of the other part can be thought of as ``half-monodromy'' around the critical points. 
Notice that a triangle lifts here no element of $\mathfrak{S}_3$: having generic parameters, the monodromy manifold exhibits no permutation-type symmetry on its coordinates. See \Cref{fig:braidvsmut} for a more extensive topographic dictionary.

\begin{remark}
    On the one hand, \Cref{thm:anconLaurent} extends to the whole parametric PVI the statement in \cite[§6.1]{CMR2016} which, limited by the CS framework, was claimed just under the $\kappa_4=1$ restriction.
    On the other hand, the argument supporting such restricted statement appears to be affected by a structural typo: formulae (6.54-56) are both involutive and easily checked not to preserve the monodromy manifold, undermining the meaning of the actions $\Tilde{\beta}_i$ for analytic continuation. Nevertheless, $\Tilde{\beta}_i$ turns out to match $\mu_i$ so that the present treatment fills the potential gap.
\end{remark}

\begin{remark}
In the spirit of confluence, we briefly explore the extent to which our cluster picture extend to other Painlevé equations---starting with PV, whose monodromy manifold can be easily obtained via the following rescalings:
\begin{equation}
\begin{aligned}
        &x_1 \ \mapsto \ \frac{x_1}{\epsilon}, \quad x_2 \ \mapsto \ \frac{x_2}{\epsilon},\\
        &\theta_1 \ \mapsto \ \frac{\theta_1}{\epsilon}, \quad \theta_2 \ \mapsto \ \frac{\theta_2}{\epsilon}, \quad \theta_3 \ \mapsto \ \frac{\theta_3}{\epsilon^2}, \quad \theta_4 \ \mapsto \ \frac{\theta_4}{\epsilon^2}.
    \end{aligned}
\end{equation}
Indeed, the dominant terms in the limit $\epsilon\rightarrow0$ result in the expected cubic
\begin{equation}\label{PV-mon-mfd}
    x_1x_2x_3+x_1^2+x_2^2-\theta_1x_1-\theta_2x_2-\theta_3x_3+\theta_4=0.
\end{equation}
Essentially, the triplet $x_1^2+x_2^2+x_3^2$ loses the third monomial without affecting the manifold's cubic term, preventing $\mu_3$ from surviving the process.
Indeed, one can easily check that only $\mu_1$ and $\mu_2$ keep preserving the cubic, provided all terms in $x_3$ are removed form their exchange polynomials.
In particular, only $\mu_{12}$ survives from the nonlinear monodromy, which descends to a mere (cyclic) action of $\integer$.
This echoes the following fact: the analytic counterpart of the above limiting procedure is a coalescence of poles, which triggers the Stokes phenomenon. Monodromy can only be a faithful measure at the unaffected pole, as the whole structure of analytic continuation breaks ``tameness'' by requiring rational maps \cite{PR2024}.

Finally, apart from PIII---whose three Dynkin types $D^{(1)}_{6,7,8}$ can each be obtained via parametric specialization of \eqref{PV-mon-mfd}---the other Painlevé equations degenerate further quadratic monomials and thus prevent \emph{any} generator of $\Gamma(2)$ from descending. 
\end{remark}   

\subsection{Discriminant equation and reduction}\label{sec:quadratic}

We now turn our attention to Conway's original topography, unveiling first the cluster nature of its arithmetic progression. For any $\mathbf{v},\mathbf{w}\in\integer^2$, local rule \eqref{Conway-rule} disguises the defining property
\begin{equation}\label{parallelogram}
    q(\mathbf{v}+\mathbf{w})+q(\mathbf{v}-\mathbf{w})=2(q(\mathbf{v})+q(\mathbf{w}))
\end{equation}
of a quadratic form $q:\integer^2\rightarrow\integer$ \cite{Veselov2023}. Indeed, integral face labels in arithmetic progression can be further interpreted as values of a binary quadratic form by going metric: meant now as the dual of the \emph{hyperbolic} Farey tessellation on the disk
\begin{equation}\label{Poin-disk}
    \mathbb{D}:=\left\{\frac{\mathtt{i}-z}{\mathtt{i}+z}\,:\,z\in\overline{\mathbb{H}}=\mathbb{H}\cup\real\cup\{\infty\}\right\},
\end{equation}
the topograph inherits an additional (rational) labeling by Farey fractions $\frac{r}{s}\in\rational$ that connects with the integral one via evaluation $q(r,s)\in\integer$.
See \Cref{fig:evaluation} for additional details expressed in topographic terms.
In particular, our reference triangle now labeled by vertices $\{\frac{1}{0},\frac{1}{1},\frac{0}{1}\}$, namely $1=r=u$ and $0=s=t$, induces the unique form 
\begin{equation}\label{init-form}
    q(x,y)=wx^2+\delta xy+ey^2, \qquad \delta:=n-(e+w),
\end{equation}
such that $q(1,0)=w,\,q(1,1)=n,\,q(0,1)=e$. In light of its recursive constructing recipe driven by $\delta$-arithmetic progression, we denote such topograph as $\pazocal{T}_q$.

\begin{figure}[t]
    \centering
    \includegraphics[width=0.4\linewidth]{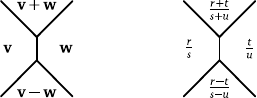}
    \caption{Local structure of Farey fractions via the signature Farey addition: vectorial on the left, projective on the right. For $\{\mathbf{e_1},\mathbf{e_2}\}$ giving a $\integer$-basis, $\mathbf{v}=r\mathbf{e_1}+s\mathbf{e_2}$ and $\mathbf{w}=t\mathbf{e_1}+u\mathbf{e_2}$. Taking evaluations, the resulting integral values $q(r-t,s-u),\,q(r,s)+q(t,u),\,q(r+t,s+u)$ are precisely in arithmetic progression as a consequence of \eqref{parallelogram}.} 
    \label{fig:evaluation}
\end{figure} 

Under specialization \eqref{Con-spec}, the master equation reduces to what we refer to as the \textbf{discriminant equation}
\begin{equation}\label{disc}
    x_1^2+x_2^2+x_3^2-2x_1x_2-2x_2x_3-2x_3x_1=\Delta.
\end{equation}
Given a topograph $\pazocal{T}_q$, it is indeed easy to check that $(w,n,e)$ solves such equation provided the master parameter matches the form's discriminant, namely $\tau=\delta^2-4ew=:\Delta$.
In particular, Conway rule is finally turned cluster as
\begin{equation}\label{q-mutation}
    n' = s=2(w+e)-n\overset{\eqref{disc}}{=\joinrel=}\frac{w^2+n^2+e^2-2ew-\Delta}{n}-n=\frac{w^2+e^2-2ew-\Delta}{n}.
\end{equation}
The rest of the paper rethinks in cluster terms the state of the art of this original topography for quadratic forms \cite{Sullivan2025}.
More precisely, we see first that $\pazocal{T}_q$ is in fact best understood as $\pazocal{T}_{[q]}$ and later reformulate the theory of reduction via mutation \eqref{q-mutation}.
Hereafter, a form is always assumed to be quadratic, binary, and integral.

\subsubsection{$\mathrm{PGL}$-equivalence}\label{sec:equiv}

Following Conway \cite[Pages 4-5]{Conway1997}, we start with
\begin{definition}
    Two forms $q_1$ and $q_2$ are \textbf{equivalent} if and only if
    \begin{equation*}
        q_1(x,y)=q_2(rx+ty,sx+uy)
    \end{equation*}
    for an invertible integral matrix
    \begin{equation*}
        M=\begin{pmatrix}r & t\\s & u\end{pmatrix}\in\mathrm{GL}_2(\integer).
    \end{equation*}
    In shorthand notation, $q_1 \sim q_2 \iff q_1=q_2 \circ M$. When additionally $\mathrm{det}(M)=1$, forms are said to be \textbf{strict-equivalent}.
\end{definition}
Seen again as a function $\integer^2\rightarrow\integer$, a form is quadratic on scalars allowing to narrow the group of equivalence down to
\begin{equation}
   \mathrm{PGL}_2(\integer)=\langle S,T,U \ | \ S^2=T^2=U^2=(SU)^2=(UT)^3=1\rangle.
\end{equation}
Let $q(x,y)=wx^2+(n-e-w)xy+ey^2$ denote a form of discriminant $\Delta$, thought of as an initial \emph{ordered} cluster $(w,n,e)$ solving its discriminant equation. In fact, any other cluster $(a,b,c)\in\integer^3$ is---as the set $\{a,b,c\}$---a triplet of values attained by $q$ as much as a distinct representative form $ax^2+(b-a-c)xy+cy^2\in[q]$. In particular, the equivalence class $[q]$ can be navigated via the following \emph{orientation-reversing} action by cluster transformations, namely (compositions of) mutations and permutations:
\begin{equation}
    \begin{matrix}
        \hfill \mathrm{PGL}_2(\integer) \hspace{.25em} & \rightarrow & \mathrm{Aut}(\pazocal{T}_{[q]}) \\[.5em]
        S=\begin{pmatrix}
            -1 & 0\\0 & 1
        \end{pmatrix} & \mapsto & \mu_2 & : & (a,b,c)  \mapsto (a,b'\!,c) \\
        T=\begin{pmatrix}
            -1 & 1\\0 & 1
        \end{pmatrix} & \mapsto & (23) & : & (a,b,c)  \mapsto (a,c,b) \\
        U=\begin{pmatrix}
            0 & 1\\1 & 0
        \end{pmatrix} & \mapsto & (13) & : & (a,b,c)  \mapsto (c,b,a)
    \end{matrix}
\end{equation}
If we restrict to $V:=UT$ and $E:=SU=US$, we match
\begin{equation}
        \mathrm{PSL}_2(\integer)=\langle E,V \ | \ E^2=V^3=1\rangle
    \end{equation}
with \emph{orientation-preserving} action
\begin{equation}
    \begin{matrix}
        \hfill \mathrm{PSL}_2(\integer) \hspace{.25em} & \rightarrow & \mathrm{Aut}(\pazocal{T}_{[q]}) \\[.5em]
        E=\begin{pmatrix}
            0 & -1\\1 & 0
        \end{pmatrix} & \mapsto & (13)\mu_2 & : & (a,b,c)  \mapsto (c,b'\!,a) \\
        V=\begin{pmatrix}
            0 & 1\\-1 & 1
        \end{pmatrix}\hfill & \mapsto & (123) & : & (a,b,c)  \mapsto (c,a,b)    
    \end{matrix}
\end{equation}
For the sake of completeness, let us observe that
\begin{equation}\label{EUV}
    \mathrm{PGL}_2(\integer)=\langle E,U,V\ | \ E^2=U^2=(EU)^2=(UV)^2=V^3=1\rangle,
\end{equation}
for $\langle U,V\rangle\simeq\mathfrak{S}_3$, while also
\begin{equation}
    \mathrm{PGL}_2(\integer)=\langle E,S,V\ | \ E^2=S^2=V^3=(ESV)^2=1\rangle.
\end{equation}
Given that the symmetric group in three elements can be presented as
\begin{equation}
    \mathfrak{S}_3=\langle X,Y \ | \ X^2=Y^3=(XY)^2=1 \rangle \ \simeq \ \mathrm{PSL}_2(\integer_2),
\end{equation}
mutation manifests by breaking the last relation but, reversing orientation on clusters much like it happens with the Markov quiver, the structure it induces trespasses $\mathrm{PSL}_2(\integer)=\integer_2*\integer_3$ into $\mathrm{PGL}_2(\integer)$. This refines the $\mathrm{SL}_2(\integer)$ approach taken by O'Sullivan \cite{Sullivan2025}: the cluster formalism naturally manifests the \emph{full} symmetry group of the topograph \cite[Page 33]{Conway1997}.

In \cite[Theorem 3.2]{Sullivan2025}, a form \eqref{init-form} is directly associated with its triplet of coefficients in notation $[w,\delta,e]$ and visualized on the topograph as an oriented edge labeled by $\delta$, see Figure 4 (b). In particular, each edge admits two orientations corresponding to the pair of strict-equivalent forms $[w,\delta,e]$ and $[e,-\delta,w]=[w,\delta,e]\circ E$.
Instead of involving edge labels, the cluster language sticks with the topograph's intrinsic face labels, identifying as above each form $q(x,y)\in[q]$ with the one cluster singled out by the standard values $q(1,0),\ q(1,1),\ q(0,1)$.
It follows that, at each vertex of $\pazocal{T}_{[q]}$ with respect to the compass local structure (\Cref{fig:compass}), the three \textbf{cardinal clusters} of type $(w,n,e)$ are in bijection with inward edges $\delta$ via the elementary relation
\begin{equation}
    \delta=n-e-w.
\end{equation}
Moreover, under full equivalence a vertex now hosts additional \textbf{anticardinal clusters} of type $(e,n,w)$, doubling the total amount to six forms. For the purpose of visualization, we split the edges of $\pazocal{T}_{[q]}$ in two halves, each one endowed with an arrow at both ends.
Thus, each vertex acquires six incident arrows, of which three inward and three outward. Each inward(outward) arrow singles out a unique form by taking the label of the $\infty$-gon it flows in(out) as northern cluster variable $n$.
An example is given in the left of \Cref{fig:maneuvers} which, in the right, visualizes the $\mathrm{PGL}_2(\integer)$-action as maneuvers in this ``decorated'' topography. Through straights and turns, these elementary moves allow to ride the topograph as an endless roadmap, connecting any two forms in the given equivalence class. In particular, the composition of matrices $S,T,U$, which is read off from the directions guiding $q_1$ to $q_2=q_1 \circ M$, results precisely in $M\in\mathrm{PGL}_2(\integer)$. 

\begin{figure}[!t]
    \centering
    \includegraphics[width=.825\linewidth]{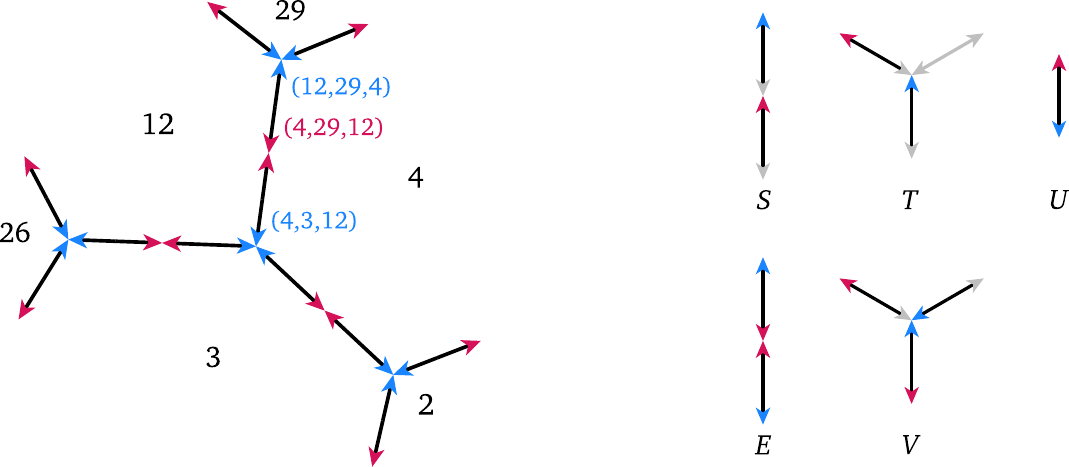}
    \caption{In the left, visualization of a topograph's equivalence class of forms, with cardinal clusters in {\color{2blue}blue} and anticardinal clusters in {\color{2red}red}. In the top right, elementary color-switching maneuvers induced by the generators of $\mathrm{PGL}_2(\integer)$: $S$ goes straight to the other half of the edge, $T$ turns at the vertex, and $U$ gives the U-turn. Analogously, the bottom right gives the color-preserving maneuvers from $\mathrm{PSL}_2(\integer)$: $E$ rotates by $\pi$ around the tail, $V$ rotates by $\pm\frac{2\pi}{3}$ around the head.}
    \label{fig:maneuvers}
\end{figure}

For completeness, the subgroup 
\begin{equation}
    \langle\mu_1,\mu_2,\mu_3\rangle \ \simeq \ \integer_2*\integer_2*\integer_2
\end{equation}
can be given a $\mathrm{PGL}_2(\integer)$ realization 
\begin{equation}
    \begin{matrix}
        \hfill \mathrm{PGL}_2(\integer) & \rightarrow & \mathrm{Aut}(\pazocal{T}_{[q]}) \\[.5em]
        VSV^2=\begin{pmatrix}
            1 & 0\\2 & -1
        \end{pmatrix}\hfill & \mapsto & \mu_1 & : & \hspace{.275em}(a,b,c)  \mapsto (a',b,c)
        \\
        V^2SV=\begin{pmatrix}
            -1 & 2\\0 & 1
        \end{pmatrix} & \mapsto & \mu_3 & : & (a,b,c)  \mapsto (a,b,c')
    \end{matrix}
\end{equation}
together with
\begin{equation}
    \langle\mu_{12},\mu_{23},\mu_{31}\rangle \ \simeq \ F_2,
\end{equation}
which recovers the standard generators for $\Gamma(2)$:
\begin{equation}
    \begin{matrix}
        \hfill \mathrm{PSL}_2(\integer) & \rightarrow & \mathrm{Aut}(\pazocal{T}_{[q]}) \\[.5em]
        \begin{pmatrix}
            1 & 0\\2 & 1
        \end{pmatrix} & \mapsto & \mu_{12}:=\mu_1\mu_2 & : & \hspace{.275em}(a,b,c)  \mapsto (a'',b',c)  \\
        \begin{pmatrix}
            1 & 2\\0 & 1
        \end{pmatrix} & \mapsto & \mu_{32}=\mu_{23}^{-1} & : & (a,b,c)  \mapsto (a,b',c'')
    \end{matrix}
\end{equation}
\begin{remark}
    Notice the ordering of tuples $(a,b,c)$ is crucial to distinguish forms within a class. This is another striking example, again coming from a concrete manifestation (number theory) of cluster phenomena much as it happened in \cite{DalMartello2024} (mathematical physics), that invites to reconsider clusters as truly ordered objects.
\end{remark}

\subsubsection{Cluster reduction}\label{sec:clustred}

In the following, we refer to O'Sullivan's uniform reduction algorithm \cite[\S6]{Sullivan2025} based on general continued fraction expansions as the \textbf{fractional algorithm}.
As was already observed by O'Sullivan in terms of the topograph's \emph{edges}, a reduced form visually manifests as a local peculiarity in topography: the well ($\Delta<0$), the lake ($\Delta=0$), the pair of river mouths ($0<\Delta=k^2$), or the cycle of river bends ($0<\Delta \neq k^2$). See \Cref{fig:local} for a visualization of these concepts in terms of the topograph's \emph{vertices}: the $\mathrm{PGL}$-viewpoint indeed invites to consider reduced representatives as \emph{unordered} clusters, visualized by vertices $\{n,e,w\}$ encompassing both possible orientations.
\begin{figure}[!t]
    \centering
    \includegraphics[width=.85\linewidth]{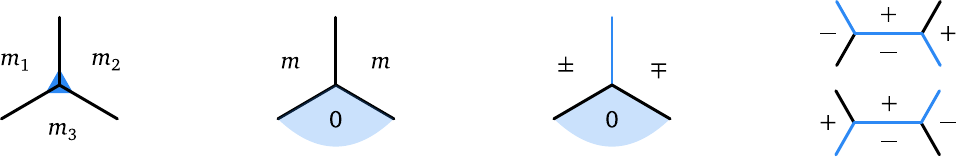}
    \caption{Bestiary of reduced vertices. From left to right: the well, with the three minimal values in the topograph; the lake, with $m$ the minimal value in the topograph; the pair of river mouths; the two types of river bends, each consisting of two adjacent vertices $v^l$ and $v^r$. In particular, {\color{2blue}blue} edges indicate the topograph's river, which we display with positive labels on top (cf. \Cref{fig:eg4}) allowing to tell mouths as well as a bend's vertices apart between left and right.}
    \label{fig:local}
\end{figure}

Echoing the fractional algorithm \cite[Definition 6.1]{Sullivan2025}, the directions guiding a form within its equivalence class to a reduced vertex are encoded by the associated root:
\begin{definition}
    Given a form $q=(a,b,c)$ of discriminant 
    \begin{equation*}
        \Delta=(b-a-c)^2-4ac=\delta^2-4ac,
    \end{equation*}
    the \textbf{root} of $q$ is the number $z_q\in\overline{\mathbb{H}}$ solving $q(z_q,1)=0$. Explicitly,
    \begin{equation}
        z_q:=\begin{cases}
            \dfrac{\sqrt{\Delta}-\delta}{2a}, \hspace{2.7em} a\neq 0,\hfill\\[1em]
            \hspace{.75em}-\dfrac{c}{\delta}, \hspace{4em} a=0\mathrm{\,\,and\,\,}\delta>0,\\[1em]
            \hspace{1.25em}\infty, \hfill a=0\mathrm{\,\,and\,\,}\delta\leq 0.
        \end{cases}
    \end{equation}
\end{definition}
Unlike the fractional algorithm, which is based on the root's continued fraction expansion, the cluster framework needs a dual codification in terms of snake graphs.
The information encoded by this zigzagging graph turns out to be twofold: the undulation pattern prescribes the sequence of mutations leading to the reduced vertex, while the counts of perfect matchings recover the Farey triangle $z_q$ belongs to. We give a preliminary
\begin{definition}
    Let $z\in\mathbb{H}\cup\real$. Then, $\mathrm{sgn}(z):=\mathrm{sgn}(\mathfrak{Re}(z))$ is the \textbf{sign} of $z$ and
    \begin{equation}
        \chi(z):=\begin{cases}
            1, \qquad \mathfrak{Re}(z)<0,\\
            0, \hfill \mathfrak{Re}(z)\geq0.
        \end{cases}
    \end{equation}    
\end{definition}

In order to cover all scenarios, namely any $\Delta\in\integer$, we give a \emph{geometric} recipe to construct the snake graph valid for any root $z_q\in\overline{\mathbb{H}}$.

Start from the hyperbolic Farey tessellation by tracing the arc connecting the disk's center $\mathtt{i}\in\mathbb{H}$ with the given root. This induces the \emph{oriented} geodesic from $\pm1$ that terminates on the Farey sum of the triangle the root belongs to (\Cref{fig:split-triangle}), with sign determined so to have the imaginary axis $\mathtt{i}\real_{>0}$ cut. In the two ambiguous cases $z_q\in\{0,\infty\}$, choose $-1$. 

Then, the snake graph acquires a diamond tile $G_k$ (\Cref{fig:diamond}) for each Farey edge $\tau_k$ the geodesic cuts, for the imaginary axis labeled as initial diagonal $\tau_0$, and the gluing rule is that of alternating relative orientation \cite[§4]{MSW2011}.
In a nutshell, one glues two overlapping tiles $G_i$ and $G_{i+1}$ alternately along the common edge or the diagonal $\tau_{i+1}$, starting with the former.
Equivalently, one unfolds the triangulated polygon singled out by the geodesic as if the snake had been embedded in the Farey tessellation folded on itself repeatedly along the inner edges---meaning that the unfolding happens \emph{southward} for a negative root (e.g. \Cref{fig:unfolding}).
In particular, the snake graph is \emph{infinite} when $z_q\in\real\setminus\rational$ while both $0$ and $\infty$ degenerate their only tile into a single tilted edge (the geodesic itself, respectively eastern and western).
Note that the resulting snake graph $\pazocal{G}(z_q)$ fits the standard literature \cite{CS2018}, up to a $45^\circ$ tilting, only when the root is positive: when negative, $G_0$ is the top tile instead of the conventional one at the bottom. In fact, our construction is, as an abstract object independent of geometry, richer than a usual snake graph (cf. \Cref{def:rattlesnake}). 

\begin{figure}[!t]
    \centering
    \includegraphics[width=.45\linewidth]{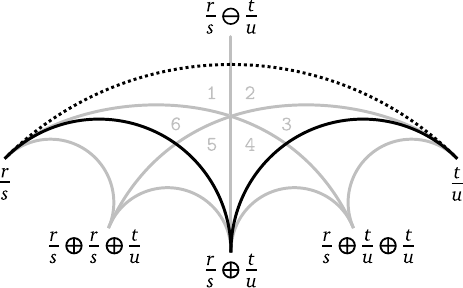}
    \caption{The six-zones splitting construction of a Farey triangle: $\frac{r}{s}\oplus\frac{t}{u}=\frac{r+t}{s+u}$ denotes the Farey addition while $\frac{r}{s}\ominus\frac{t}{u}=\frac{r-t}{s-u}$ denotes the Farey subtraction, the latter well-defined provided the sign of the resulting fraction is attributed to the numerator. With the exclusion of the dotted top edge (unless $r=u=0$ and $s=t=1$), the geodesic induced by a root $z_q\in\mathbb{H}$ on the triangle must terminate in the vertex $\frac{r}{s}\oplus\frac{t}{u}$. When $z_q\in\real$, the geodesic simply terminates on the root itself.}
    \label{fig:split-triangle}
\end{figure}

\begin{figure}[!t]
    \centering
    \includegraphics[width=.4\linewidth]{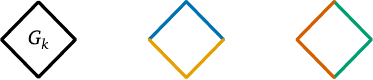}
    \caption{From left to right: the building block of a snake graph, the diamond tile; a tile's \textbf{northern} edges in {\color{3blue}blue} and \textbf{southern} edges in {\color{3ochre}yellow}; a tile's \textbf{western} edges in {\color{3red}red} and \textbf{eastern} edges in {\color{3green}green}. Observe that tiles are oriented like squares in the standard literature, e.g. \cite[\S2.2]{Ovenhouse2023} provided the letter $L$ in our treatment is translated to an ``upward'' $U$.}
    \label{fig:diamond}
\end{figure}

\begin{figure}[!t]
    \centering
    \includegraphics[width=.975\linewidth]{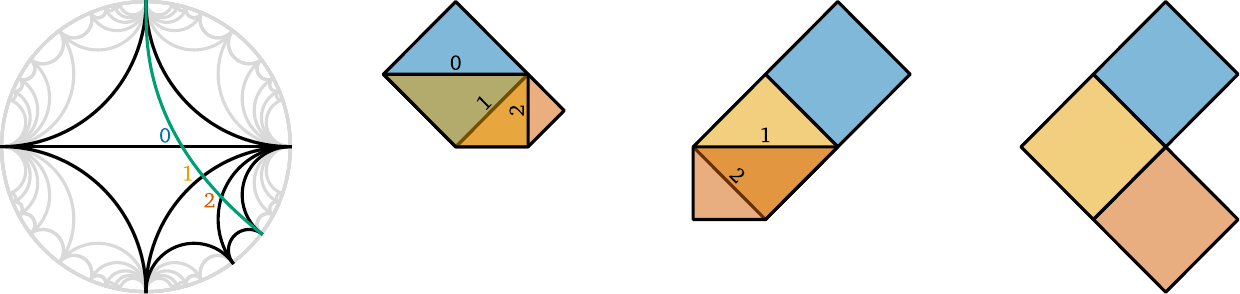}
    \caption{Unfolding the triangulated polygon singled out by the {\color{3green}green} geodesic that terminates in $\nicefrac{-1}{3}$: after ``rectifying'' the polygon's Farey edges, one unfolds first along $\tau_0$, revealing the snake graph's initial {\color{3blue}blue} tile $G_0$, and then along $\tau_1$ to complete the snake graph with word $W(\pazocal{G}(\nicefrac{-1}{3}))=SLR$.}
    \label{fig:unfolding}
\end{figure}

Next, encode the shape of $\pazocal{G}(z_q)$ via the following \textbf{word} in alphabet $\{L,R,S\}$: 
\begin{equation}\label{W(G)}
    W(\pazocal{G}(z_q)):=S^{\chi(z_q)}w,
\end{equation}
for $w$ a word in alphabet $\{L,R\}$ codifying whether a tile is glued to the next westward ($L$) or eastward ($R$) along the geodesic.
In particular, the word is empty at $\infty$ or when $0\leq\mathfrak{Re}(z_q)\leq1$. 

Finally, extract the sequence of mutations as follows: starting with $2$ as initial index $i\in\integer_3$,
\begin{equation}
    \begin{matrix}
        S & \mapsto & \mu_{i},\\
        L & \mapsto & \mu_{i-1},\\
        R & \mapsto & \mu_{i+1}.\\
    \end{matrix}
\end{equation}
\Cref{ex:negative} gives a detailed run of the resulting \textbf{cluster algorithm} for a negative complex root.

\begin{definition}
    Given a form $q$ of discriminant $\Delta$, the \textbf{tuple of reduced vertices} for the topograph $\pazocal{T}_{[q]}$ consists in the following data: a singleton $\{m_1,m_2,m_3\}$ or $\{m,0,m\}$ when $\Delta\leq0$, a $2$-tuple $(\{-,0,+\},\{+,0,-\})$ with left mouth first and right mouth last when $\Delta=k^2$, a $n$-tuple $(v^l_1,v^r_1;v^l_2,v^r_2;\ldots;v^l_n,v^r_n)$ with bends separated by semicolons when $0<\Delta\neq k^2$. 
\end{definition}

Before proving the cluster algorithm's output is precisely such tuple, we need the following
\begin{definition}
    Let $w=w_1 \cdots w_k$ be a word in the alphabet $\{L,R\}$ and $\overline{L}:=R$, $\overline{R}:=L$.
    The word
    \begin{equation}
        \overline{w}:=\overline{w}_1\overline{w}_2\overline{w}_3\overline{w}_4 \cdots \overline{w}_k
    \end{equation}
    is the \textbf{opposite} of $w$.  
    Depending on parity, the word 
    \begin{equation}
        w^*:=\begin{cases}
            \overline{w}_1w_2\overline{w}_3w_4 \cdots \overline{w}_{2n-1}w_{2n},\\
            \overline{w}_1w_2\overline{w}_3w_4 \cdots w_{2n}\overline{w}_{2n+1},
        \end{cases}
    \end{equation}
    is the \textbf{dual} of $w$.
    In particular, the word $w_*:=\overline{w}^*$ gives the \textbf{codual} of $w$. 
\end{definition}

\begin{theorem}\label{thm:reducLaurent}
    Let $q=(w,n,e)$ be a form with nonpositive discriminant. Then, the sequence of mutations attached to its root $z_q$ terminates on the well or the unique lake. When $0<\Delta\neq k^2$, the (infinite) sequence runs along the periodic river cycling through the river bends. When $\Delta=k^2$, the sequence may halt (at the left mouth) or terminate at a lake. In the latter case, an additional alternating sequence of mutations canonically reaches the right mouth.
\end{theorem}
\begin{proof}
    Assume first the root is positive. As a word, $\pazocal{G}(z_q)$ is none other than the dual of the root's \textbf{cutting sequence} \cite[§7]{Sullivan2026}, itself a word encoded by numbers $c_i\in\naturali$ with $c_1\geq0$: 
    \begin{equation*}
        W(\pazocal{G}(z_q))=(L^{c_1}R^{c_2}L^{c_3}\cdots [L\text{\,or\,}R]^{c_n})^*.
    \end{equation*}
    This can be seen in the usual local terms of a cluster $(w,n,e)$: indeed, $\mu_e$ goes west and $\mu_w$ goes east but, being a mutation orientation-reversing, at every other step prescribed by the sequence of mutations such cardinal mismatch is compensated.

    When the root is an irrational real, the cutting sequence matches the (infinite) directions of the fractional algorithm \cite[Theorem 7.3]{Sullivan2026}. Once on the river, each bend is detected by a sign-change in the mutated variable.
    
    For a rational root, by the same Theorem $z_q=[c_1,c_2,\ldots,c_n,1]$ so that, when $\Delta=0$, the cutting sequence equals the directions of the fractional algorithm minus the last turn. Since the reduced form $q_0=[w_0,\delta_0,e_0]$ lays on the lakeshore (as the edge $\delta_0$), the cluster algorithm terminates on the lake.  
    If the discriminant is a perfect square, the cluster framework manifests a numerical obstruction that turns lakes into ``black holes'': once a sequence of mutations hits any lake, it can only proceed along its shore (the ``event horizon'') in that any attempt to leave would trigger a division by zero. It follows that, when the fractional algorithm hits first
    \begin{itemize}\itemsep0em\vspace{-.25em}
        \item the left lake, the cluster algorithm halts at the left river mouth;
        \item the river, the cluster algorithm terminates at the right river mouth;
        \item the right lake, the cluster algorithm terminates when 0 (from the right lake) appears as a value. The nearest (right) river mouth is reached with the additional alternating sequence of mutations along the shore.
    \end{itemize}\vspace{-.25em}
    Such additional sequence canonically descends to the mouth by decreasing absolute values.
    
    When the root is genuinely complex, namely
    \begin{equation}
        z_q\overset{\text{\cite{Sullivan2026}}}{=\joinrel=}[a_1,a_2,\ldots,a_r+z_0]\in\mathbb{H},
    \end{equation}
    a tetrachotomy manifests: referring to the zones in \Cref{fig:split-triangle}, the cutting sequence equals
    \begin{itemize}\itemsep0em\vspace{-.25em}
        \item the directions of the fractional algorithm minus the last turn inside \texttt{3}, \texttt{5};
        \item all the directions of the fractional algorithm inside \texttt{2};
        \item $L^{a_1}R^{a_2}L^{a_3}\cdots [L\text{ or }R]^{a_r}$ inside \texttt{1};
        \item $L^{a_1}R^{a_2}L^{a_3}\cdots [L\text{ or }R]^{a_r-1}$ inside \texttt{4}, \texttt{6}.
    \end{itemize}\vspace{-.25em}
    In the first case, $\delta_0=m_3-m_1-m_2>0$ is indeed outward for the vertex well, requiring the last turn to be ignored.
    In the second case, $\delta_0=m_3-m_1-m_2<0$ is inward for the vertex well and the cluster algorithm needs even the last turn to reach the reduced vertex.
    In the last two cases, the directions need a final reversal (a letter $S$ in the original paper instead of our $E$) which the cluster algorithm is free from: in the third(fourth) case, $\delta_0$ is indeed outward(inward). 
    In particular, when $z_q$ belongs to the triangle having $1$ as Farey sum, $q$ is already part of a reduced vertex and no mutation is prescribed.
    
    Therefore, when the root has nonnegative real part, the cluster algorithm terminates precisely on a compatible reduced vertex from \Cref{fig:local}.
    
    Assume now a negative root. Then, the fractional algorithm is limited to left or right turns and needs a cumbersome backward motion $L^{-k}R,\ k\geq1$, to invert direction and align with the optimal reduction path.
    In cluster terms, such inversion manifests simply as a single index-preserving mutation, namely $\mu_n$ for the initial cluster $q=(w,n,e)$ corresponding to the misaligned edge $\delta=n-e-w$.
    The root of the resulting form $q \circ S=(w,s,e)$ is given by the reflection of $z_q$ fixing the disk's \emph{horizontal} axis:
    \begin{equation}
        z_{q\circ S}=-\mathfrak{Re}(z_q)+\mathtt{i}\mathfrak{Im}(z_q).
    \end{equation}
    In particular, $z_{q\circ S}$ is now positive (much as $q \circ S$ is now aligned as an edge of the topograph $\pazocal{T}_{q\circ S}$) and leads to a reduced vertex by the previous argument.
    The word it induces can be computed as we prescribed in the geometric constructing recipe of the snake graph: being reflected extends from the roots $z_q$ and $z_{q\circ S}$ to the corresponding snake graphs $\pazocal{G}(z_q)$ and $\pazocal{G}(z_{q\circ S})$, whose words $w$ in the restricted alphabet $\{L,R\}$ indeed agree.  
\end{proof}

\begin{remark}
    On the one hand, the numerical obstruction mentioned in the proof implies \emph{both} mouths need to be considered reduced in order to cluster-rule whether two forms with the same discriminant $k^2$ are equivalent. A mouth is recovered from the other along the river with Conway's rule---at least for the first step entering the river.
    Essentially, this shortcoming is equivalent to the inefficient need \cite[Page 27]{Sullivan2025} of a second run for the fractional algorithm to achieve a unique reduced representative.
    In the exceptional case of the so-called \textbf{weir}
    \begin{figure}[!ht]
        \centering
        \includegraphics[height=4.75em]{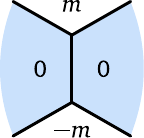}
    \end{figure}
    \\the upper(lower) vertex act as the left(right) mouth---for a river of length zero---and the two can be here mapped via the only possible mutation.
    
    On the other hand, notice the cardinality of a cycle doubles for $\mathrm{PGL}$: under flips of orientation, a simply \cite[Definition 6.14]{Sullivan2025} reduced form $q_0=[a_0,\delta_0,c_0]$ is reversed to $[c_0,-\delta_0,a_0]$.  
\end{remark}
\begin{remark}
    For an abstract snake graph $\pazocal{G}$, words $W(\pazocal{G})$ are made of either straights $LL\cdots L$, $RR\cdots R$, or zigzags $LRL\cdots[L\text{\,or\,}R]$, $RLR\cdots[L\text{\,or\,}R]$. In the topograph (fractional viewpoint), the former run along the boundary of a fixed $\infty$-gon while the latter keep jumping, after at most two edges, to adjacent $\infty$-gons. In the Farey tessellation (cluster viewpoint), these roles are swapped and the resulting phenomenology, which echoes the setting of associahedra in finite type, illuminates why duality is the one gateway to the language of mutations. On the one hand, straights turn into zigzags yielding alternating sequences $\mu_j\mu_k\mu_j\cdots[\mu_j\text{\,or\,}\mu_k]$ which border the $\infty$-gon labeled by the one fixed cluster value $x_i$.
    On the other hand, a zigzag rectifies into a straight yielding (mod 3) monotone sequences like $\mu_1\mu_2\mu_3\mu_1\cdots[\mu_i\text{\,or\,}\mu_j\text{\,or\,}\mu_k]$ or $\mu_1\mu_3\mu_2\mu_1\cdots[\mu_i\text{\,or\,}\mu_j\text{\,or\,}\mu_k]$. Monotonicity prevents a value from surviving more than three consecutive clusters: indeed, the triplet $\mu_i\mu_j\mu_i$ would have value $x_k$ in all its \emph{four} clusters (2 adjacent edges involve 3 vertices).
\end{remark}

For strict-equivalence, a canonical reduced cluster $q_0^\mu\in\integer^3$ must be singled out from any reduced vertex.
Denoting by the same $R$ the matrix defined in \cite[(2.4)]{Sullivan2026} and by $q_0$ or $\{q_{0,i}\}_{i\in I}$ the output of the fractional algorithm, we prescribe the following natural choices: 

\begin{itemize}\itemsep0em\vspace{-.5em}
        \item at the well, match $q_0$ (up to reversal $E$) by cyclically permuting $\mathrm{max}\{m_1,m_2,m_3\}$ to middle cluster value---additionally swapping by $U$ when the sequence of mutations is of odd length;
        \item at the lake, take the cluster $(m,0,m)=q_0 \circ R^{-1}$;
        \item at a mouth, take $(-,0,+)=q_0 \circ R^{-1}$ when left and $(+,0,-)$ when right;
        \item at a river bend detected by a mutation $\mu_j: z_j \mapsto z_j',\ \mathrm{sgn}(z_j')\neq\mathrm{sgn}(z_j)$, take $(-,z_j',+)=q_{0,k}$ for some $k\in I$.  
    \end{itemize}\vspace{-.5em}
In particular, a bend yields a single reduced cluster, halving the tuple's cardinality when $\Delta=k^2$. 

\begin{corollary}
    Two forms share the same tuple of reduced vertices if and only if they are equivalent. Moreover, they additionally share the tuple of reduced clusters if and only if they are strict-equivalent.
\end{corollary}
\begin{proof}
    The tuple of reduced vertices is uniquely determined, as is the sequence of mutations being minimal with respect to length. Therefore, two forms sharing the tuple of reduced vertices must appear on the same topograph \emph{up to orientation}. Being canonical, the tuple of reduced clusters, whose entries are now strict-equivalent to the initial form, is uniquely determined as well. Thus, two forms sharing the tuple of reduced clusters must appear on the \emph{very} same topograph.
\end{proof}

As a consequence of \Cref{thm:reducLaurent}, the cluster algorithm endows reduction with Laurentness, proving the remaining part B of \Cref{thm:intro}:
\begin{corollary}
    The procedure of reduction of binary integral quadratic forms, for any discriminant $\Delta\in\integer\neq k^2$, exhibits the Laurent phenomenon: starting from a reduced form $ax^2+\delta xy+cy^2$ corresponding to the solution $\mathbf{z}^0=(z_1^0,z_2^0,z_3^0)=(a,a+\delta+c,c)\in(\integer^*)^3$ of the discriminant equation, any other equivalent form will correspond to a solution whose entries express as Laurent polynomials of the initial $\mathbf{z}^0$.
\end{corollary}

\subsubsection{Rattles and rationals}

We conclude the paper focusing on the combinatorics of perfect matchings in snake graphs, starting with an abstraction of the full information captured by the geometry-driven $\pazocal{G}(z)$:
\begin{definition}\label{def:rattlesnake}
    A \textbf{rattle} is an edge $e=\{u,v\}$ in a snake graph $\pazocal{G}$ such that $\mathrm{deg}(u)=\mathrm{deg}(v)\leq2$. The pair $(\pazocal{G},e)$ gives a \textbf{rattlesnake}. 
\end{definition}
In the following, rattles acquire cardinal adjectives as in \Cref{fig:diamond}.
Let $\pazocal{G}_n=(G_0,\ldots,G_{n-1})$ denote an abstract snake graph consisting of $n$ tiles---that is, whose diamond tiles are glued (northward) either westward or eastward starting from the initial $G_0$. For $\pazocal{G}_0$, namely a diamond's single edge, the rattle is obviously unique and either western or eastern. With the exception of all four edges being rattles in $\pazocal{G}_1$, there are just two rattles otherwise: exactly one southern and one northern.
Such pair of rattles is well-defined even in a single tile, provided an identification based on the following abstract analogue of \eqref{W(G)}:
\begin{equation}
    W(\pazocal{G},e)=\begin{cases}
        w(e), \hfill \text{$e$ southern,}\\
        Sw(e), \qquad \text{$e$ northern,}
    \end{cases}
\end{equation}
where $w(e)$ is a word in alphabet $\{L,R\}$ that, starting from the tile hosting the rattle $e$, prescribes the cardinal gluing instructions of the snake graph $\pazocal{G}$.
\begin{definition}
    Two rattlesnakes $(\pazocal{G}_1,e_1)$ and $(\pazocal{G}_2,e_2)$ are said to be \textbf{equivalent} if both $\pazocal{G}_1=\pazocal{G}_2$ and $W(\pazocal{G}_1,e_1)=W(\pazocal{G}_2,e_2)$.
    In shorthand notation, $(\pazocal{G}_1,e_1)\sim(\pazocal{G}_2,e_2) \iff [\pazocal{G}_1,e_1]=[\pazocal{G}_2,e_2]$.
\end{definition}

In fact, the snake graph $\pazocal{G}(z)$ from \Cref{sec:clustred} disguises the abstract rattlesnake $[\pazocal{G}(z),e(z)]$ singled out by the unique rattle 
\begin{equation}
    e(z) \begin{cases}
        \text{southern}, \qquad \chi(z)=0,\\
        \text{northern}, \hfill \chi(z)=1.
    \end{cases}
\end{equation}
In particular, the word attached to $G(z)$ via geometry is recovered in purely graph-theoretic terms as $W[\pazocal{G}(z),e(z)]$.
As anticipated in the Introduction, $\pazocal{G}(z)$ knows in full the unique rational number that, meant as a Farey sum, singles out precisely the one Farey triangle $z$ belongs to. Defining a \textbf{perfect matching} in $[\pazocal{G},e]$ as a perfect matching in $\pazocal{G}$ that includes the rattle, we indeed get

\begin{theorem}
    Let $z\in\overline{\mathbb{H}}$ induce a geodesic that terminates in $\pm\frac{r}{s}\in\rational\cup\{\infty\}$. Then, the number of perfect matchings in $[\pazocal{G}(z),e(z)]$ equals $\mathrm{max}\{r,s\}$, namely
    \begin{equation}
        \mathrm{Match}[\pazocal{G}(z),e(z)]=\begin{cases}
            r, \quad \mathrm{if} \hspace{.5em} |\frac{r}{s}|\geq1,\\
            s, \quad \mathrm{if} \hspace{.5em} |\frac{r}{s}|\leq1.
        \end{cases}
    \end{equation}
    In particular,
    \begin{equation}
        \mathrm{Match}(\pazocal{G}(z))=r+s.
    \end{equation}
\end{theorem}
\begin{proof}
    It suffices to assume $z=\pm\frac{r}{s}\in\rational_{\geq1}$, namely $r\geq s$.    
    
    When $z=+\frac{r}{s}=[a_1,\ldots,a_n,1]$ is positive,
    \begin{equation}
        \mathrm{Match}[\G{\nicefrac{r}{s}},\e{\nicefrac{r}{s}}]=r.
    \end{equation}
    Indeed, forcing the rattle in matchings is equivalent to ignoring the first tile in $\pazocal{G}(z)$; in turn, $\pazocal{G}(z)\setminus G_0$ recovers Çanakçı-Schiffler's snake graph $\pazocal{G}[a_1,\ldots,a_n,1]$ for positive continued fractions, whose count of perfect matchings is known to yield the numerator \cite[Theorem 3.4]{CS2018}: choosing $e_0$ as the west side,
    \begin{equation*}
        W(\pazocal{G}[a_1,\ldots,a_n,1])\overset{\text{\cite{Ovenhouse2023}}}{=\joinrel=}(R^{a_1-1}L^{a_2}\cdots[L\mathrm{\,or\,}R]^{a_n})^*=(L^{a_1-1}R^{a_2}\cdots[R\mathrm{\,or\,}L]^{a_n})_*=W(\G{\nicefrac{r}{s}}\setminus G_0).
    \end{equation*}
    Since $\frac{s}{r}=[0,a_1,\ldots,a_n,1]$ and $L^0R^{a_1}L^{a_2}\cdots [L\text{\,or\,}R]^{a_n}=\overline{L^{a_1}\cdots [R\text{\,or\,}L]^{a_n}}$,
    \begin{equation}\label{opp-G}
       W(\G{\nicefrac{s}{r}})=(L^0R^{a_1}\cdots [L\text{\,or\,}R]^{a_n})^*=(L^{a_1}\cdots [R\text{\,or\,}L]^{a_n})_*=\overline{W}(\G{\nicefrac{r}{s}}),
    \end{equation}
    so that $W(\G{\nicefrac{s}{r}}\setminus G_0)=\overline{W}(\G{\nicefrac{r}{s}}\setminus G_0)$.
    Being the perfect matchings of opposite snake graphs in obvious bijection, $\mathrm{Match}[\G{\nicefrac{s}{r}},\e{\nicefrac{s}{r}}]=r$. 
    In topographic terms, \eqref{opp-G} is a consequence of the Farey tessellation's reflection symmetry fixing the \emph{vertical} axis.    
    
    For the whole snake graph, $\frac{r+s}{s}=[1+a_1,\ldots,a_n,1]$ yields
    \begin{equation}
        W(\G{\nicefrac{r+s}{s}}\setminus G_0)=(R^{a_1}L^{a_2}R^{a_1}\cdots [L\text{\,or\,}R]^{a_n})^*=\overline{W}(\G{\nicefrac{r}{s}})=W(\G{\nicefrac{s}{r}}).
    \end{equation}
    Analogously, $\frac{r+s}{r}=[1,a_1,\ldots,a_n,1]$ yields
    \begin{equation}
        W(\G{\nicefrac{r+s}{r}}\setminus G_0)=(L^{a_1}R^{a_2}\cdots [L\text{\,or\,}R]^{a_n})^*=W(\G{\nicefrac{r}{s}})=\overline{W}(\G{\nicefrac{s}{r}}).
    \end{equation}
    In particular, it follows that
    \begin{equation}
        \mathrm{Match}(\G{\nicefrac{r}{s}})=\mathrm{Match}(\G{\nicefrac{s}{r}})=r+s.
    \end{equation}
    
    When $z<0$, it suffices to notice that
    \begin{equation}\label{vertword}
        W(\pazocal{G}(z))=SW(\pazocal{G}(-z))
    \end{equation}
    together with the letter $S$ being neutral to matchings. In topographic terms, \eqref{vertword} is a consequence of the Farey tessellation's reflection symmetry fixing the \emph{horizontal} axis.
\end{proof}

In practice, the rattle codifies the fraction's sign and magnitude at once: on the one hand, it is southern when the fraction is positive or northern when negative; on the other hand, it is westward when the fraction is greater than $1$ or eastward otherwise.

Thus, a noteworthy final consequence follows: 
\begin{corollary}
    Let $z\in\rational\cup\{\infty\}$. Then, 
    \begin{equation*}
        |z|=
        \begin{cases}
            \dfrac{\mathrm{Match}[G(z),e(z)]}{\mathrm{Match}(G(z))-\mathrm{Match}[G(z),e(z)]}, \qquad \mathrm{if} \hspace{.5em} |z|\geq1,\\[2em]
            \dfrac{\mathrm{Match}(G(z))-\mathrm{Match}[G(z),e(z)]}{\mathrm{Match}[G(z),e(z)]}, \qquad \mathrm{if} \hspace{.5em} |z|\leq1,
        \end{cases}
    \end{equation*}
    and the sign of $z$ is captured by $e(z)$: positive when the rattle is southern and negative otherwise. 
    In particular, the set of rationals together with $\infty$ is in bijection with isoclasses of rattlesnakes:
    \begin{equation}
    \begin{matrix}
            \epsilon\!\left(\frac{\mathrm{Match}[G,e]}{\mathrm{Match}(G)-\mathrm{Match}[G,e]}\right)^{\varepsilon} & \mathrel{\reflectbox{\ensuremath{\mapsto}}} & [G,e] \\[.5em]
            \rotatebox[origin=c]{90}{$\ni$} & & \rotatebox[origin=c]{90}{$\ni$} \\[.5em]
        \hspace{.55em}\rational\cup\infty & \overset{1:1}{\longleftrightarrow} & \hspace{1em} \{\mathrm{rattlesnakes}\}/\!\sim\\[.5em]
            \rotatebox[origin=c]{90}{$\in$} & & \rotatebox[origin=c]{90}{$\in$}\\[.5em]
        z & \mapsto & [\pazocal{G}(z),e(z)]
    \end{matrix}
    \end{equation}
    where
    \begin{equation}
        \epsilon=\begin{cases}
            -1, \quad \text{e $\mathrm{northern}$},\\
            +1, \quad \text{e $\mathrm{southern}$},
        \end{cases}
        \qquad
        \varepsilon=\begin{cases}
            -1, \quad \text{e $\mathrm{eastern}$}, \\
            +1, \quad \text{e $\mathrm{western}$}.
        \end{cases}
    \end{equation}
\end{corollary}

\begin{example}[$\Delta=-31<0$]\label{ex:negative}
    For the cluster $\hat{q}=(16,200,103)$, the root $z_{\hat{q}}=\frac{-81+\mathtt{i}\sqrt{31}}{32}$ is a genuine complex number located inside the triangle singled out by the Farey sum $(\nicefrac{-3}{1})\oplus(\nicefrac{-2}{1})$, see \Cref{fig:plot}.    
    The induced geodesic, connecting the pair of Farey fractions $-1$ and $\nicefrac{-5}{2}$, yields the rattlesnake displayed in \Cref{fig:eghat} whose word reads as $SRLL$.    
    Keeping track of the index at each mutation step, we thus get the following sequence of clusters\vspace{-.75em}
    \begin{equation*}
        \begin{matrix}
            _{i\,=\,2} & & _{i\,=\,2} & & _{i\,=\,3} & & _{i\,=\,2} & & _{i\,=\,1} \\[.25em] \hline \\[-1em]
            (16,200,103) & \overset{\mu_2}{\longmapsto} & (16,38,103) & \overset{\mu_3}{\underset{i+1}{\longmapsto}} & (16,38,5) & \overset{\mu_2}{\underset{i-1}{\longmapsto}} & (16,4,5) & \overset{\mu_1}{\underset{i-1}{\longmapsto}} & (2,4,5)
        \end{matrix}
    \end{equation*}
    culminating in the vertex $\{2,4,5\}$. This is indeed reduced in that it matches the vertex well for the topograph $\pazocal{T}_{\hat{q}}$ of $\hat{q}=16x^2+81xy+103y^2$ \cite[Figure 9]{Sullivan2025}.    
    Passing to strict-equivalence, the fractional algorithm yields the reduced form $\hat{q}_0=2x^2+xy+4y^2$ that matches the reduced cluster after reversal: $\hat{q}_0^\mu=(4,5,2)=\hat{q}_0 \circ E=E[2,1,4]=[4,-1,2]$.
    En passant, we confirm the rattle $e(z_{\hat{q}})$ is northwestern, that
    \begin{equation*}
        \mathrm{Match}[\pazocal{G}(z_{\hat{q}}),e(z_{\hat{q}})]=\mathrm{Match}(\square\!\square\!\square)=5,
    \end{equation*}
    and that $\mathrm{Match}(\pazocal{G}(z_{\hat{q}}))=7=5+2$ all as expected.
    
    Let us now take $\check{q}=(41,20,4)$ and verify that $\check{q}\sim\hat{q}$. This time, $z_{\check{q}}=\frac{25+\mathtt{i}\sqrt{31}}{82}$ is positive and the geodesic terminates in the Farey sum $\frac{1}{3}\oplus\frac{0}{1}=\frac{1}{4}$, see again \Cref{fig:plot}.
    \begin{figure}[t]
        \centering
        \includegraphics[width=0.825\linewidth]{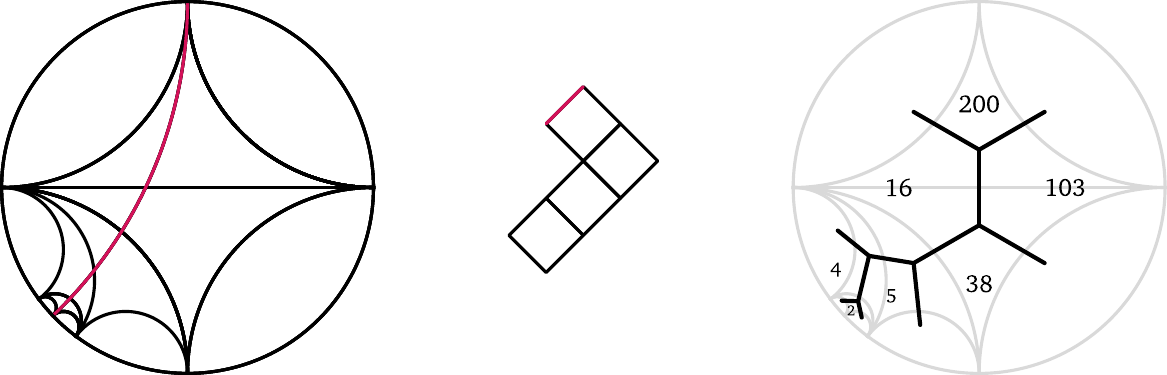}
        \caption{From left to right: geodesic, rattlesnake, and topograph for {\color{2red}$z_{\hat{q}}$}.}
        \label{fig:eghat}
    \end{figure}
    \begin{figure}[t]
        \centering
        \includegraphics[width=0.825\linewidth]{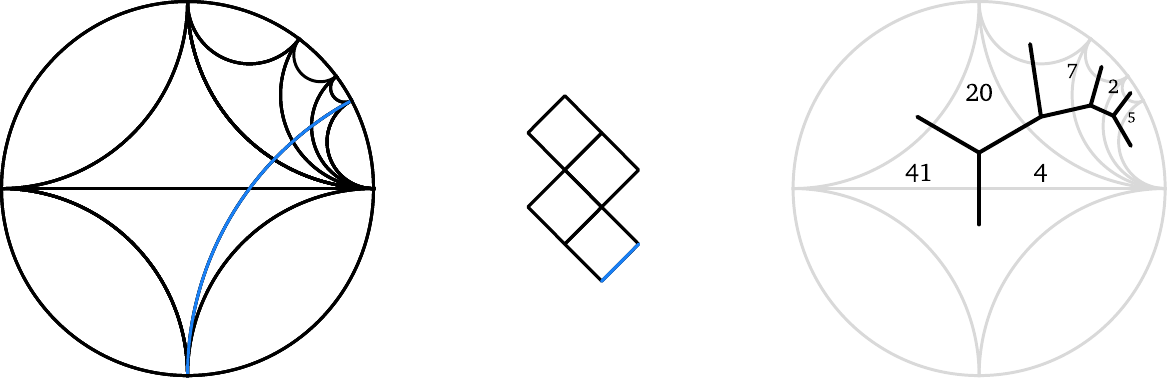}
        \caption{From left to right: geodesic, rattlesnake, and topograph for {\color{2blue}$z_{\check{q}}$}.}
        \label{fig:egcheck}
    \end{figure}
    The snake graph reads now as the zigzag $LRL$ in \Cref{fig:egcheck}, whose count of perfect matchings is $5$; forcing the (southeastern) rattle in, the number of perfect matchings reduces to the expected $4$. The resulting sequence of clusters\vspace{-.75em}
    \begin{equation*}
        \begin{matrix}
            _{i\,=\,2} & & _{i\,=\,1} & & _{i\,=\,2} & & _{i\,=\,1} \\[.25em] \hline \\[-1em]
            (41,20,4) & \overset{\mu_1}{\underset{i-1}{\longmapsto}} & (7,20,4) & \overset{\mu_2}{\underset{i+1}{\longmapsto}} & (7,2,4) & \overset{\mu_1}{\underset{i-1}{\longmapsto}} & (5,2,4)
        \end{matrix}
    \end{equation*}
    culminates in the reduced vertex $\{5,2,4\}$ which, as a set, coincides with the one previously found for $\hat{q}$, confirming the two forms to be $\mathrm{PGL}$-equivalent.
    On the contrary, the reduced clusters differ: being $\mu_1\mu_2\mu_1$ of odd length, $\check{q}_0^\mu=U(4,5,2)=(2,5,4)=\check{q}_0\neq\hat{q}_0^\mu$. Indeed, $\check{q}$ and $\hat{q}$ appears on oppositely oriented topographs (Figures \ref{fig:eghat} and \ref{fig:egcheck}), revealing them not to be further equivalent in the strict sense.

    We conclude the example comparing words with directions from the fractional algorithm. For the complex residues $\hat{z}_0=\frac{-1-\mathtt{i}\sqrt{31}}{8}$ and $\check{z}_0=\frac{1-i\sqrt{31}}{8}$,
    \begin{equation*}
        z_{\hat{q}}=-3+\frac{1}{2+\hat{z}_0}=[-3,2+\hat{z}_0], \hspace{3em} z_{\check{q}}=\frac{1}{3+\check{z}_0}=[0,3+\check{z}_0].
    \end{equation*}
    Directions in the former case read as $L^{-3}R^2$. Taking
    \begin{equation*}
        z_{\hat{q}\circ S}=\tfrac{81+\mathtt{i}\sqrt{31}}{32}=\left[2,1,1+\tfrac{-1+\mathtt{i}\sqrt{31}}{8}\right]
    \end{equation*}
    instead and comparing with the snake graph via duality confirms a zone \texttt{5} behavior:
    \begin{equation*}
        (LLRL)^*=\underbrace{RLL}_w\hspace{-.2em}L.
    \end{equation*}
    In the latter case, all directions $L^0R^3=RRR$ are needed, confirming a zone \texttt{2} behavior:
    \begin{equation*}
        (RRR)^*=LRL=W(\pazocal{G}(z_{\check{q}})).
    \end{equation*}
    For completeness, a cluster in zone \texttt{4}, with respect to the triangle having Farey sum $\frac{1}{2}$, would be $(16,4,5)$ of root $\frac{17+\mathtt{i}\sqrt{31}}{32}$ (\Cref{fig:plot}): the fractional algorithm prescribes $RLE$ instead of the single cluster instruction $L$. Indeed, $\mu_1(16,4,5)=(2,4,5)$.     
\end{example}

\begin{figure}[!ht]
        \centering
        \includegraphics[width=.875\linewidth]{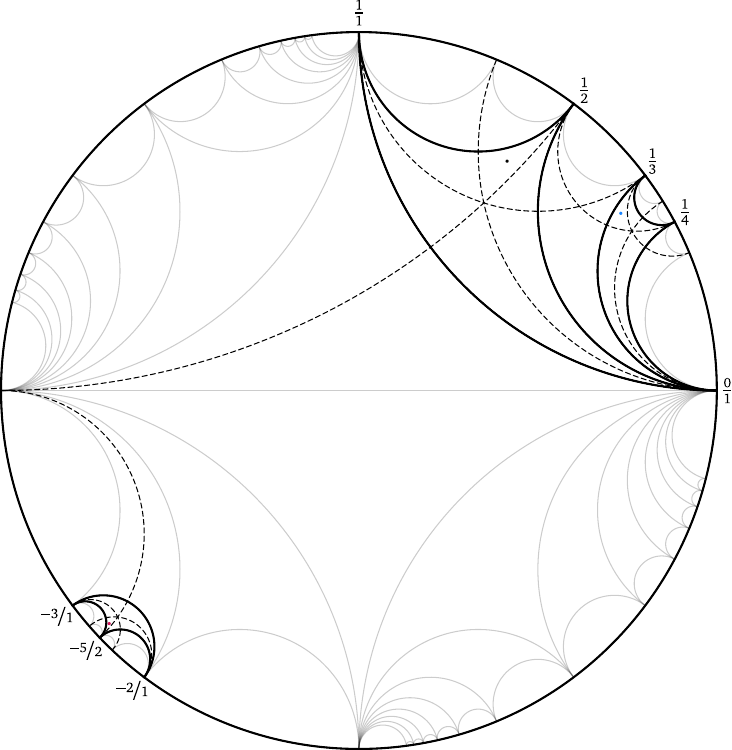}
        \caption{The roots {\color{2red}$z_{\hat{q}}=\frac{-81+\mathtt{i}\sqrt{31}}{32}$} and {\color{2blue}$z_{\check{q}}=\frac{25+\mathtt{i}\sqrt{31}}{82}$}, together with $\frac{17+\mathtt{i}\sqrt{31}}{32}$, located on the hyperbolic Farey tessellation of the disk $\mathbb{D}$.}
        \label{fig:plot}
    \end{figure}

\begin{figure}[t]
        \centering
        \includegraphics[width=.375\linewidth]{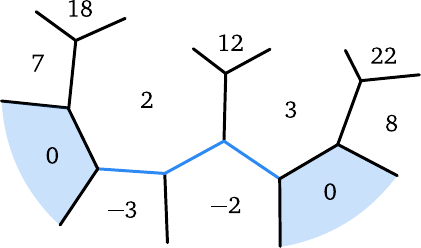}
        \caption{Conway's river for the quadratic form $18x^2-23xy+7y^2$.}
        \label{fig:eg4}
    \end{figure}

\begin{example}[$\Delta=25=5^2$]\label{ex:perfectsquare}
    For $q_1=(18,2,7)$, $W(\pazocal{G}(\nicefrac{7}{9}))=L^2RL^2$. However, after two mutations we must halt at vertex $\{0,2,-3\}$, which is verified to be the left mouth in $\pazocal{T}_{q_1}$ (\Cref{fig:eg4}). Proceeding with Conway's rule for $0$ as north label, we reach the right mouth $\{-2,0,3\}$. Explicitly,
    \begin{equation*}
        \begin{matrix}
            (18,2,7) & \overset{\mu_1}{\longmapsto} & (0,2,7) & \overset{\mu_3}{\longmapsto} & (0,2,-3) & \overset{\mu_1}{\longmapsto} & \bm{\oslash}\hfill & & & \\[.75em]
            & & & & \hspace{1.5em}\rotatebox[origin=c]{-90}{$\longmapsto$}\eqref{Conway-rule} & & & & & \\[.5em]
            & & & & (-2,2,-3) & \overset{\mu_3}{\longmapsto} & (-2,2,3) & \overset{\mu_2}{\longmapsto} & (-2,0,-3).
        \end{matrix}
    \end{equation*}
    
    For $q_2=(2,12,3)$, $W(\pazocal{G}(\nicefrac{-1}{2}))=SL=SW(\pazocal{G}(\nicefrac{1}{2}))=S\overline{W}(\pazocal{G}(2))$ and the algorithm runs in full till the expected right mouth $\{0,-2,3\}$:
    \begin{equation*}
        \begin{matrix}
            (2,12,3) & \overset{\mu_2}{\longmapsto} & (2,-2,3) & \overset{\mu_1}{\longmapsto} & (0,-2,3).
        \end{matrix}
    \end{equation*}
    
    For $q_3=(8,3,22)$, $W(\pazocal{G}(2))=R$ and the algorithm terminates at the right lakeshore $\{8,3,0\}$, see \Cref{fig:eg4}. Since $8>3$, we proceed with $\mu_1$ obtaining the vertex $\{-2,3,0\}$ whose sign-change confirms it to be the right mouth:
    \begin{equation*}
        \begin{matrix}
            (8,3,22) & \overset{\mu_3}{\longmapsto} & (8,3,0) & \overset{\mu_1}{\longmapsto} & (-2,3,0).
        \end{matrix}
    \end{equation*}
    This triplet of forms exhausts all possible behaviors when the discriminant is a perfect square.  
    \end{example}



{\small
\bibliographystyle{abbrv}
\bibliography{references}
}

\end{document}